\documentclass[11pt, a4paper]{article}

\usepackage[utf8]{inputenc}
\usepackage{amsmath, amsfonts, amssymb, amsthm, cases, csquotes, enumitem, hyperref, mathrsfs, mathtools, tikz, tikz-cd}
\usepackage[style=alphabetic]{biblatex}

\theoremstyle{definition}

\newtheorem{definition}{Definition}[section]
\newtheorem{theorem}{Theorem}[section]
\newtheorem{proposition}{Proposition}[section]
\newtheorem{corollary}{Corollary}[section]
\newtheorem{lemma}{Lemma}[section]
\newtheorem{example}{Example}[section]
\newtheorem{conjecture}{Conjecture}[section]

\usepackage[left=2cm,right=2cm,top=2cm,bottom=2cm,bindingoffset=0cm]{geometry}

\newcommand*{\isoarrow}[1]{\ar[#1,"\rotatebox{90}{\(\sim\)}"]}

\newcommand{\Addresses}{{
		\bigskip
		\footnotesize
		
		P.~Kosenko, \textsc{Department of Mathematics, Higher School of Economics, Moscow, Usacheva str. 6}\par\nopagebreak
		\textit{E-mail address}, P.~Kosenko: \texttt{prkosenko@edu.hse.ru}\par\nopagebreak
		\textsc{Department of Mathematics, University of Toronto, 40 St. George St., Toronto, ON, Canada}\par\nopagebreak
		\textit{E-mail address: } \texttt{petr.kosenko@mail.utoronto.ca}
}}

\title{Homological dimensions of smooth crossed products}
\author{Petr Kosenko}

\begin{document}
	\maketitle
	
	
	\begin{abstract}
		In this paper we provide upper estimates for the global projective dimensions of smooth crossed products $\mathscr{S}(G, A; \alpha)$ for $G = \mathbb{R}$ and $G = \mathbb{T}$ and a self-induced Fr\'echet-Arens-Michael algebra $A$. In order to do this, we provide a powerful generalization of methods which are used in the works of Ogneva and Helemskii.
	\end{abstract}
	
	\section*{Introduction}
	\indent
	There are numerous papers dedicated to homological properties of smooth crossed products of Fr\'echet algebras and C*-algebras, see \cite{schweitzer1993dense}, \cite{phillips1994representable}, \cite{2004math.....10596M}, \cite{2011arXiv1111.2154G}, or \cite{neshveyev2014smooth}, for example.
	
	However, it seems that nothing is known about homological dimensions of smooth crossed products. In the paper \cite{2017arXiv171206177K} we provided the estimates for homological dimensions of holomorphic Ore extensions and smooth crossed products by $\mathbb{Z}$ of unital $\hat{\otimes}$-algebras, and in this paper we show that the methods of the author's previous works and the paper \cite{ogneva1984homological} can be adapted to smooth crossed products by $\mathbb{R}$ and $\mathbb{T}$.
	
	The idea behind the estimates lies in the construction of admissible $\widehat{\Omega}^1$-like sequences for the required \textit{non-unital} algebras. What do we mean by that? Recall the definition of a bimodule of relative 1-forms:
	
	\begin{definition}
		Let $A$ be an algebra and $X$ be an $A$-bimodule. A linear map $d : A \longrightarrow X$ is called an $A$-\textit{derivation} if
		\[
		d(ab) = d(a) \circ b + a \circ d(b)
		\]
		for every $a, b \in A$.
	\end{definition}
	
	\begin{definition}
		Let $R$ be a unital $\hat{\otimes}$-algebra, and let $A$ denote a unital $R$-$\hat{\otimes}$-algebra (see Definition \ref{non-unital}). A pair $\left(\widehat{\Omega}^1_R(A), d_A\right)$, which consists of a $A$-$\hat{\otimes}$-bimodule $\widehat{\Omega}^1_R(A)$ and a continuous $R$-derivation $d_A : A \longrightarrow \widehat{\Omega}^1_R(A)$, is called the \textit{bimodule of relative 1-forms} of $A$, if this pair is universal in the following sense:
		
		for every $A$-$\hat{\otimes}$-bimodule $M$ and a continuous $R$-derivation $D : A \longrightarrow M$ there exists a unique continuous $A$-$\hat{\otimes}$-bimodule homomorphism $\widetilde{D} : \widehat{\Omega}^1_R(A) \longrightarrow M$ such that $D = \widetilde{D} \circ d_A$.
		\[
		\begin{tikzcd}
		\widehat{\Omega}^1_R(A) \ar[r, "\widetilde{D}"] & M \\
		A \ar[u, "d_A"] \ar[ur, "D"']
		\end{tikzcd}
		\]
	\end{definition}

	This construction is a topological version of a construction presented in \cite{CuntzQuil}. It is not hard to prove that $\widehat{\Omega}^1_R(A)$ is a well-defined object, moreover, this bimodule is a part of an extremely useful admissible sequence. The following theorem is the topological version of \cite[Proposition 2.5]{CuntzQuil}.
	
	\begin{theorem}[\cite{PirkAM}, Proposition 7.2]
		Let $R$ be a unital $\hat{\otimes}$-algebra and let $A$ denote a unital $R$-$\hat{\otimes}$-algebra. Then there exists a sequence which splits in the categories $A$-$R$-$\hat{\otimes}$-\textbf{mod} and $R$-$A$-$\hat{\otimes}$-\textbf{mod}:
		\begin{equation}
			\label{seq1}
			\begin{tikzcd}
			0 \ar[r] & \widehat{\Omega}^1_R(A) \ar[r, "j"] & A \hat{\otimes}_R A \ar[r, "m"] & A \ar[r] & 0,
			\end{tikzcd}
		\end{equation}
		where $m(a \otimes b) = ab.$ In particular, this sequence is admissible.
	\end{theorem}

	In the paper \cite{2017arXiv171206177K} we utilized the sequence \ref{seq1} in order to obtain the upper estimates for the homological dimensions of different types of non-commutative Ore-like extensions. This sequence proves to be quite useful because in the unital case for the extensions we study it turns out that $\widehat{\Omega}^1_R(A) \cong A \hat{\otimes}_R A$ as a $R$-$\hat{\otimes}$-module.
		
	However, when $G = \mathbb{R}$ or $G = \mathbb{T}$, then for a Fr\'echet-Arens-Michael algebra $A$ the algebras $\mathscr{S}(G, A; \alpha)$ are, in general, not unital. Nevertheless, we managed to obtain the exact sequences for these algebras, which look similar to the \eqref{seq1}, and which allowed us to derive the upper estimates for the global projective dimensions of $\mathscr{S}(\mathbb{R}, A; \alpha)$ and $\mathscr{S}(\mathbb{T}, A; \alpha)$. 
	
	We conjecture that estimates for homological dimensions should look as follows:
	\begin{conjecture}
		Let $A$ be a Fr\'echet-Arens-Michael algebra (not necessarily unital) with a smooth $m$-tempered action $\alpha$ of $\mathbb{R}$ or $\mathbb{T}$ on $A$. Denote the left (projective) global dimension by $\text{dgl}$. Then for $G = \mathbb{R}$ or $G = \mathbb{T}$ we have
		\[
			\text{dgl}(A) \le \text{dgl}(\mathscr{S}(G, A; \alpha)) \le \text{dgl}(A) + 1.
		\]
	\end{conjecture}

	The main results of this paper are Theorems \ref{main theorem}, \ref{main theorem2}, \ref{main theorem3}, \ref{main theorem4}. In particular, we have proven the weak form of the above conjecture.
	
	\begin{theorem}
		Let $A$ be a Fr\'echet-Arens-Michael algebra, which satisfies the following condition: the multiplication map $m : A \hat{\otimes}_A A \longrightarrow A$ is a $A$-$\hat{\otimes}$-bimodule isomorphism. Also let $\alpha$ denote a smooth $m$-tempered action of $\mathbb{R}$ or $\mathbb{T}$ on $A$. Denote the left (projective) global dimension by $\text{dgl}$. Then for $G = \mathbb{R}$ or $G = \mathbb{T}$ we have
		\[
			\text{dgl}(\mathscr{S}(G, A; \alpha)) \le \max \{ \text{dgl}(A), 1 \} + 1
		\]
	\end{theorem}
	
	\section{Preliminaries}
	\subsection{Notation}
	\textbf{Remark.} All algebras in this paper are defined over the field of complex numbers and assumed to be associative. Unlike in the paper \cite{2017arXiv171206177K}, here we allow the algebras are to be \textit{non-unital}.
	
	\begin{definition}
		A \textit{Fr\'echet} space is a complete metrizable locally convex space.
	\end{definition}
	
	Let us introduce some notation (see \cite{Helem1986} and \cite{PirkVdB} for more details). Denote by \textbf{LCS}, \textbf{Fr} the categories of complete locally convex spaces, Fr\'echet spaces, respectively. Also we will denote the category of vector spaces by $\textbf{Lin}$.
	
	For a locally convex Hausdorff space $E$ we will denote its completion by $\tilde{E}$. Also for locally convex Hausdorff spaces $E, F$ the notation $E \hat{\otimes} F$ denotes the completed projective tensor product of $E, F$.
	
	By $A_+$ we will denote the unitization of an algebra $A$. By $A^{\text{op}}$ we denote the opposite algebra.
	
	%
	
	\begin{definition}
		A complete locally convex algebra with jointly continuous multiplication is called a $\hat{\otimes}$\textit{-algebra}.
	\end{definition} 
	A $\hat{\otimes}$-algebra with the underlying locally compact space which is a Fr\'echet space is called a \textit{Fr\'echet algebra}.
	
	\begin{definition}
		A locally convex algebra $A$ is called $m$-convex if the topology on it can be defined by a family of submultiplicative seminorms.
	\end{definition}
	
	\begin{definition}
		A complete locally $m$-convex algebra is called an \textit{Arens-Michael algebra}.
	\end{definition}

	\begin{definition}
		Let $A$ be a $\hat{\otimes}$-algebra and let $M$ be a complete locally convex space which is also a left $A$-module. Also suppose that the natural map $A \times M \rightarrow M$ is jointly continuous. Then we will call $M$ a left $A$-$\hat{\otimes}$\textit{-module}. In a similar fashion we define right $A$-$\hat{\otimes}$\textit{-modules} and $A$-$B$-$\hat{\otimes}$\textit{-bimodules}.
		
		A $\hat{\otimes}$-module over a Fr\'echet algebra which is itself a Fr\'echet space is called a Fr\'echet $A$-$\hat{\otimes}$-module.
	\end{definition} 
	For arbitrary $\hat{\otimes}$-algebras $A, B$ we denote
	\[
	\begin{aligned}
	A\text{-}\textbf{mod} &= \text{the category of left } A\text{-}\hat{\otimes}\text{-modules}, \\
	\textbf{mod}\text{-}A &= \text{the category of right } A\text{-}\hat{\otimes}\text{-modules}, \\
	A\text{-}\textbf{mod}\text{-}B &=  \text{the category of } A\text{-}B\text{-}\hat{\otimes}\text{-bimodules}. \\
	\end{aligned}
	\]
	
	For unital $\hat{\otimes}$-algebras $A, B$ we denote
	\[
	\begin{aligned}
	A\text{-}\textbf{unmod} &= \text{the category of unital left } A\text{-}\hat{\otimes}\text{-modules}, \\
	\textbf{unmod}\text{-}A &= \text{the category of unital right } A\text{-}\hat{\otimes}\text{-modules}, \\
	A\text{-}\textbf{unmod}\text{-}B &=  \text{the category of unital } A\text{-}B\text{-}\hat{\otimes}\text{-bimodules}. \\
	\end{aligned}
	\]
	
	
	Let $A$ be a $\hat{\otimes}$-algebra, and consider a complex of $A$-$\hat{\otimes}$-modules:
	$$
	\dots \xrightarrow{d_{n+1}} M_{n+1} \xrightarrow{d_n} M_n \xrightarrow{d_{n-1}} M_{n-1} \xrightarrow{d_{n-2}} \dots,
	$$
	then we will denote this complex by $\{M, d\}$.
	
	\begin{definition}
		Let $A$ be a $\hat{\otimes}$-algebra and consider a left $A$-$\hat{\otimes}$-module $Y$ and a right $A$-$\hat{\otimes}$-module $X$.
		\begin{enumerate}[label=(\arabic*)]
			\item A bilinear map $f : X \times Y \longrightarrow Z$, where $Z \in \textbf{LCS}$, is called $A$-balanced if $f(x \circ a, y) = f(x, a \circ y)$ for every $x \in X, y \in Y, a \in A$.
			\item A pair $(X \hat{\otimes}_A Y, i)$, where $X \hat{\otimes}_A Y \in \textbf{LCS}$, and $i : X \times Y \longrightarrow X \hat{\otimes}_A Y$ is a continuous $A$-balanced map, is called the \textit{completed projective tensor product of} $X$ and $Y$, if for every $Z \in \textbf{LCS}$ and  continuous $A$-balanced map $f : X \times Y \longrightarrow Z$ there exists a unique continuous linear map $\tilde{f} : X \hat{\otimes}_A Y \longrightarrow Z$ such that $f = \tilde{f} \circ i$.
		\end{enumerate}
	\end{definition}

	\subsection{Projectivity and homological dimensions} 
	The following definitions shall be given in the case of left modules; the definitions in the cases of right modules and bimodules are similar, just use the following category isomorphisms:
	for unital $A, B$ we have
	\[
	\begin{aligned}
	\textbf{unmod}\text{-}A &\simeq A^{\text{op}}\text{-}\textbf{unmod} & A\text{-}\textbf{unmod}\text{-}B &\simeq (A \hat{\otimes} B^{\text{op}}) \text{-}\textbf{unmod}
	\end{aligned}
	\]

	\noindent
	Let $A$ be a unital $\hat{\otimes}$-algebra.
	
	\begin{definition}
		A complex of $A$-$\hat{\otimes}$-modules $\{M, d\}$ is called \textit{admissible} $\iff$ it splits in the category $\mathbf{LCS}$. A morphism of $A$-$\hat{\otimes}$-modules $f : X \rightarrow Y$ is called \textit{admissible} if it is one of the morphisms in an admissible complex.
	\end{definition}
	
	\begin{definition}
		An additive functor $F : A\text{-}\textbf{mod} \rightarrow \textbf{Lin}$ is called  \textit{exact} $\iff$ for every admissible complex $\{M, d\}$ the corresponding complex $\{F(M), F(d)\}$ in $\textbf{Lin}$ is exact.
	\end{definition}
	
	\begin{definition}
		\label{unitaldef}
		Suppose that $A$ and $B$ are unital $\hat{\otimes}$-algebras.
		\begin{enumerate}[label=(\arabic*)]
			\item A module $P \in A\text{-}\textbf{unmod}$ is called \textit{projective} $\iff$ the functor $\text{Hom}_A(P, -)$ is exact. 
			\item A module $X \in A\text{-}\textbf{unmod}$ is called \textit{free} $\iff X$ is isomorphic to $A \hat{\otimes} E$ for some $E \in \mathbf{LCS}$.
		\end{enumerate}
	\end{definition}

	Now we consider the general, non-unital case. Let $A$ be a $\hat{\otimes}$-algebra. Any left $\hat{\otimes}$-module over an algebra $A$ can be viewed as a unital $\hat{\otimes}$-module over $A_+$, in other words, the following isomorphism of categories takes place:
	\[
	A\text{-}\textbf{mod} \cong A_+\text{-}\textbf{unmod}, \quad A\text{-}\textbf{mod}\text{-}B \cong A_+ \hat{\otimes} B_+^{\text{op}}\text{-}\textbf{unmod}.
	\]
	By using this isomorphism we can define projective and free modules in the non-unital case.
	\begin{definition}
		\label{nonunitaldef}
		Suppose that $A$ and $B$ are $\hat{\otimes}$-algebras.
		\indent
		\begin{enumerate}[label=(\arabic*)]
			\item A module $P \in A\text{-}\textbf{mod}$ is called \textit{projective} $\iff$ the module $P$ is projective in the category $A_+\text{-}\textbf{unmod}$ 
			\item A module $X \in A\text{-}\textbf{mod}$ is called \textit{free} $\iff X$ is isomorphic to $A_+ \hat{\otimes} E$ for some $E \in \mathbf{LCS}$.
		\end{enumerate}
	\end{definition}
    As it turns out, there is no ambiguity, a unital module is projective in the sense of the Definition \ref{unitaldef} if and only if it is projective in the sense of the Definition \ref{nonunitaldef}.
	\begin{definition}
		Let $X \in A\text{-}\textbf{mod}$. Suppose that $X$ can be included in a following admissible complex:
		$$
		0 \leftarrow X \xleftarrow{\varepsilon} P_0 \xleftarrow{d_0} P_1 \xleftarrow{d_1} \dots  \xleftarrow{d_{n-1}} P_n \leftarrow 0 \leftarrow 0 \leftarrow \dots,
		$$
		where every $P_i$ is a projective module. Then we will call the complex $\{P, d\}$, where $$\{P, d\} = 0 \leftarrow P_0 \xleftarrow{d_0} P_1 \xleftarrow{d_1} \dots \xleftarrow{d_{n-1}} P_n \leftarrow 0,$$ the \textit{projective resolution} of $X$ of \textit{length} $n$. By definition, the length of an unbounded resolution equals $\infty$.
	\end{definition}
	
	This allows us to define the notion of a derived functor in the topological case, for example, see \cite[ch 3.3]{Helem1986}. In particular, $\text{Ext}_A^k(M, N)$ and $\text{Tor}^A_k(M, N)$ are defined similarly to the purely algebraic situation.
	
	\begin{definition}
		Consider an arbitrary module $M \in A\text{-}\textbf{mod}$. Then following number is well-defined:
		\[
		\begin{aligned}
		\text{dh}_A(M) &  = \min\{ n \in \mathbb{Z}_{\ge 0} : \text{Ext}^{n+1}_A (M, N) = 0 \text{ for every } N \in A\text{-}\textbf{mod} \} = \\ & = \{\text{the length of a shortest projective resolution of }M\} \in \{-\infty\} \cup [0, \infty].
		\end{aligned}
		\]
		It is called the \textit{projective (homological) dimension} of $M$. 
	\end{definition}
	
	\begin{definition}
		Let $A$ be a $\hat{\otimes}$-algebra. Then we can define the following invariants of $A$:
		\[
		\text{dgl}(A) = \sup \{ \text{dh}_A(M) : M \in A\text{-}\textbf{mod}\} - \textit{the left global dimension of } A.
		\]
		\[
		\text{dgr}(A) = \sup \{ \text{dh}_A(M) : M \in \textbf{mod}\text{-}A\} - \textit{the right global dimension of } A.
		\]
	\end{definition}

	\subsection{Algebra of rapidly decreasing functions}
	Recall the definition of the space of rapidly decreasing functions on $\mathbb{R}^n$.
	\begin{definition}
		For $n > 0$ define the Fr\'echet space
		\[
		\mathscr{S}(\mathbb{R}^n) := \{ f : \mathbb{R}^n \rightarrow \mathbb{C} : \left\| f \right\|_{k,l} = \sup_{x \in \mathbb{R}^n} | x^k D^l(f) | < \infty \text{ for all } k, l \in \mathbb{Z}_{\ge 0}^n \},
		\]
		where $x^k = x_1^{k_1} \dots x_n^{k_n}$ and $D^l(f) = \frac{\partial^{l_1}}{\partial x_1^{l_1}} \dots \frac{\partial^{l_n}}{\partial x_n^{l_n}} f$. The topology on $\mathscr{S}(\mathbb{R}^n)$ is defined by the system $\{ \left\| f \right\|_{k,l} : k, l \in \mathbb{Z}_{\ge 0}^n \}$.
	\end{definition}

	There are two natural ways to define the multiplication on $\mathscr{S}(\mathbb{R}^n)$:
	\[
	(f \cdot g)(x) = f(x) g(x) \quad \text{(pointwise product)}
	\]
	\[
	(f * g)(x) = \int_{\mathbb{R}^n} f(y) g(x - y) \text{d} y. \quad \text{(convolution product)}
	\]
	
	The following theorem is well-known.
	
	\begin{theorem}
		Fix $n \in \mathbb{N}$.
		\begin{enumerate}[label=(\arabic*)]
			\item $(\mathscr{S}(\mathbb{R}^n), \cdot)$ is a Fr\'echet-Arens-Michael algebra.
			\item The Fourier transform induces an isomorphism of Arens-Michael algebras
			\[
			\mathcal{F}_n : (\mathscr{S}(\mathbb{R}^n), \cdot) \longrightarrow (\mathscr{S}(\mathbb{R}^n), *),
			\]
			\[
			\mathcal{F}_n(f)(x) = \int_{\mathbb{R}^n} f(y) e^{-2\pi i \left\langle x, y \right\rangle} \text{d}y_1 \dots \text{d}y_n
			\]
		\end{enumerate}
	\end{theorem}
	\begin{proof}
		\begin{enumerate}[label=(\arabic*)]
			\item The proof is very similar to the proof that $C^\infty(\mathbb{R}^n)$ is a Fr\'echet-Arens-Michael algebra, which can be found in \cite[Section 4.4.(2)]{book:338950}.
			\item See \cite[Theorem 8.22, Corollary 8.28]{book:5272} for the proof.
		\end{enumerate}
	\end{proof}

	From now on we will write $\mathscr{S}(\mathbb{R}^n)$ instead of $(\mathscr{S}(\mathbb{R}^n), \cdot)$ and $\mathscr{S}(\mathbb{R}^n)_{\text{conv}}$ instead of $(\mathscr{S}(\mathbb{R}^n), *)$.
	

	

\subsection{\texorpdfstring{$\widehat{\Omega}^1$}{Omega1}-like admissible sequences for \texorpdfstring{$\mathscr{S}(\mathbb{R})$}{SR}}
In order to obtain the homological dimensions of $\mathscr{S}(\mathbb{R}^n)$ in the paper \cite{ogneva1984homological}, Helemskii and Ogneva used a simple and natural $\widehat{\Omega}^1$-like admissible sequence for $\mathscr{S}(\mathbb{R})$. It was constructed using Hadamard's lemma.
	\begin{lemma}[Hadamard's lemma]
		\label{Hadamard's lemma}
		Let $f \in \mathscr{S}(\mathbb{R}^n)$, such that $f(0, x_2, \dots, x_n) = 0$ for all \\ $(x_2, \dots, x_n) \in \mathbb{R}^{n-1}$. Then there exists a function $g \in \mathscr{S}(\mathbb{R}^n)$ such that
		\[
		f(x_1, \dots, x_n) = x_1 g(x_1, \dots, x_n).
		\]
		More generally, suppose that $f(x) = 0$ on a hyperplane in $\mathbb{R}^n$ defined by the equation $a_1 x_1 + \dots + a_n x_n = 0$. Then there exists $g \in \mathscr{S}(\mathbb{R}^n)$ such that
		\[
		f(x_1, \dots, x_n) = (a_1 x_1 + \dots + a_n x_n) g(x_1, \dots, x_n).
		\]
	\end{lemma}

		

	Recall that $\mathscr{S}(\mathbb{R}^2)$ admits the following structure of a $\mathscr{S}(\mathbb{R})$-$\hat{\otimes}$-bimodule:
	\[
	(\varphi \cdot f)(x, y) = \varphi(x)f(x, y), \quad (f \cdot \varphi)(x, y) = f(x, y) \varphi(y) 
	\]
	for any $\varphi \in \mathscr{S}(\mathbb{R})$, $f \in \mathscr{S}(\mathbb{R}^2)$, $x, y \in \mathbb{R}$.
	
	The Theorem 1.1 gives a similar $\mathscr{S}(\mathbb{R})_{\text{conv}}$-$\hat{\otimes}$-bimodule structure on $\mathscr{S}(\mathbb{R}^2)_{\text{conv}}$.
	
	\begin{proposition}[\cite{ogneva1984homological}, Proposition 3]
		The following diagram is commutative, moreover, the rows of the diagram are short exact sequences of $\mathscr{S}(\mathbb{R})$-$\hat{\otimes}$-bimodules which split in the categories $\mathscr{S}(\mathbb{R})$-\textbf{mod} and \textbf{mod}-$\mathscr{S}(\mathbb{R})$:
		\begin{equation}
			\label{exactseq1}
			\begin{tikzcd}
			0 \arrow[r] &  \mathscr{S}(\mathbb{R}^2) \arrow[r, "j"] & \mathscr{S}(\mathbb{R}^2) \arrow[r,  "\pi"] & \mathscr{S}(\mathbb{R}) \arrow[r]  & 0 \\
			0 \arrow[r] &  \mathscr{S}(\mathbb{R}) \hat{\otimes} \mathscr{S}(\mathbb{R}) \arrow[r, "k"] \isoarrow{u}  & \mathscr{S}(\mathbb{R}) \hat{\otimes} \mathscr{S}(\mathbb{R}) \isoarrow{u}  \arrow[r, "m"] & \mathscr{S}(\mathbb{R}) \arrow[r] \arrow[u, "\text{Id}"] & 0
			\end{tikzcd}
		\end{equation}
		where
		\[
		\begin{aligned}
			j(f)(x, y) &= (x - y) f(x, y) \quad & \text{for all } f \in \mathscr{S}(\mathbb{R}^2), \\
			\pi(f)(x) &= f(x, x) \quad & \text{for all } f \in \mathscr{S}(\mathbb{R}^2) \\
			k(f \otimes g) &= fx \otimes g - f \otimes gx \quad & \text{for all } f \in \mathscr{S}(\mathbb{R}^2), \\
			m(f \otimes g) &= fg \quad & \text{for all } f \in \mathscr{S}(\mathbb{R}^2).
		\end{aligned}
		\]
	\end{proposition}

	Let us restate the above proposition for $\mathscr{S}(\mathbb{R})_{\text{conv}}$. First of all, we will formulate a lemma which can be considered as the ``Fourier dual'' to Hadamard's lemma.
	
	\begin{lemma}
		Let $f \in \mathscr{S}(\mathbb{R}^n)$ such that $\int_{\mathbb{R}} f(t, x_2, \dots, x_n) \text{d}t = 0$ for any $(x_2, \dots, x_n) \in \mathbb{R}^{n-1}.$ Then there exists a function $g \in \mathscr{S}(\mathbb{R}^n)$ satisfying
		\[
		f(x_1, \dots, x_n) = \frac{\partial}{\partial x_1} g(x_1, \dots, x_n).
		\]
		More generally, if there is a vector $v = (v_1, \dots, v_n) \in \mathbb{R}^n$ such that the integral $\int_{\mathbb{R}} f(x + tv) \text{d} t = 0$ for any $x \in \mathbb{R}^n$, then there exists a function $g \in \mathscr{S}(\mathbb{R}^n)$ satisfying
		\[
		f(x_1, \dots, x_n) = \left( \sum_{i = 1}^n v_i \frac{\partial}{\partial x_i} \right) g(x_1, \dots, x_n).
		\]
	\end{lemma}
	
	\begin{proposition}
		The following diagram is commutative, moreover, the rows of the diagram are short exact sequences of $\mathscr{S}(\mathbb{R})_{\text{conv}}$-$\hat{\otimes}$-bimodules which split in the categories $\mathscr{S}(\mathbb{R})_{\text{conv}}$-\textbf{mod} and \textbf{mod}-$\mathscr{S}(\mathbb{R})_{\text{conv}}$:
		\begin{equation}
		\label{exactseqconv1}
		\begin{tikzcd}
		0 \arrow[r] &  \mathscr{S}(\mathbb{R}^2)_{\text{conv}} \arrow[r, "j"] & \mathscr{S}(\mathbb{R}^2)_{\text{conv}} \arrow[r,  "\pi"] & \mathscr{S}(\mathbb{R})_{\text{conv}} \arrow[r]  & 0 \\
		0 \arrow[r] &  \mathscr{S}(\mathbb{R})_{\text{conv}} \hat{\otimes} \mathscr{S}(\mathbb{R})_{\text{conv}} \arrow[r, "k"] \isoarrow{u} & \mathscr{S}(\mathbb{R})_{\text{conv}} \hat{\otimes} \mathscr{S}(\mathbb{R})_{\text{conv}} \arrow[r, "m"] \isoarrow{u} & \mathscr{S}(\mathbb{R})_{\text{conv}} \arrow[r] \arrow[u, "\text{Id}"] & 0
		\end{tikzcd}
		\end{equation}
		where
		\[
		\begin{aligned}
		j(f)(x, y) &= \left(\frac{\partial}{\partial x} - \frac{\partial}{\partial y} \right) f(x, y) \quad & \text{for all } f \in \mathscr{S}(\mathbb{R}^2), \\
		\pi(f)(x) &= \int_{\mathbb{R}} f(y, x-y) \text{d} y \quad & \text{for all } f \in \mathscr{S}(\mathbb{R}^2). \\
		k(f \otimes g) &= f' \otimes g - f \otimes g' \quad & \text{for all } f \in \mathscr{S}(\mathbb{R}^2), \\
		m(f \otimes g) &= f*g \quad & \text{for all } f \in \mathscr{S}(\mathbb{R}^2).
		\end{aligned}
		\]
	\end{proposition}

	In the next section we will show that the diagram \ref{exactseqconv1} can be generalized if we replace $\mathscr{S}(\mathbb{R})$ with smooth crossed products of Fr\'echet-Arens-Michael algebras by $\mathbb{R}$ and $\mathbb{T}$.

\section{\texorpdfstring{$\widehat{\Omega}^1$}{Omega1}-like admissible sequences for smooth crossed products}
\subsection{Smooth m-tempered actions and smooth crossed products}

\begin{definition}
	Let $E$ be a Hausdorff topological vector space. For a function $f : \mathbb{R}^n \rightarrow E$ and $x \in \mathbb{R}$ we denote
	\[
	\frac{\partial f}{\partial x_i}(x) := \lim_{h \rightarrow 0} \frac{f(x_1, \dots, x_i + h, \dots, x_n) - f(x_1, \dots, x_i, \dots, x_n)}{h}.
	\]
\end{definition}

\begin{definition}
	Let $X$ be a Fr\'echet space with topology, generated by a sequence of seminorms $\{ \left\| \cdot \right\|_m : m \in \mathbb{N}\}$.\\
	(1) The space $\mathscr{S}(\mathbb{T}^n, X) := C^\infty(\mathbb{T}^n, X)$ is a Fr\'echet space with respect to the system $$\left\lbrace \left\| f \right\|_{k, m} = \sup\limits_{x \in \mathbb{T}^n} \left\| D^k(f)(x) \right\|_m : k \in \mathbb{Z}^n_{\ge 0}, m \in \mathbb{N} \right\rbrace.$$
	(2) Define the following space:
	\[
	\mathscr{S}(\mathbb{R}^n, X) = \left\lbrace f : \mathbb{R}^n \rightarrow X : \left\| f \right\|_{k, l, m} := \sup_{x \in \mathbb{R}^n} \left\| x^l D^k(f)(x) \right\|_m < \infty \text{ for all } k, l \in \mathbb{Z}_{\ge 0}^n, m \in \mathbb{N}\right\rbrace,
	\]
	where $D^k(f) = \frac{\partial^{k_1}}{\partial x_i^{k_1}} \dots \frac{\partial^{k_n}}{\partial x_i^{k_n}}f$. (assuming $\mathbb{T} = \mathbb{R} / \mathbb{Z}$.) The topology on $\mathscr{S}(\mathbb{R}^n, X)$ is defined by the system $\{ \left\| f \right\|_{k, l, m} : k, l \in \mathbb{Z}_{\ge 0}^n, m \in \mathbb{N} \}$.
\end{definition}

The following proposition can be proven in the same way as in the \cite[Chapter 11.2]{book:338950}. 

\begin{proposition}
	Let $A$ be a Fr\'echet space. Then the natural maps
	\[
	\begin{aligned}
	\mathscr{S}(\mathbb{R}^n) \hat{\otimes} A \rightarrow \mathscr{S}(\mathbb{R}^n, A),& \quad f \otimes a \mapsto (x \mapsto f(x)a), \\ \mathscr{S}(\mathbb{T}^n) \hat{\otimes} A \rightarrow \mathscr{S}(\mathbb{T}^n, A),& \quad f \otimes a \mapsto (x \mapsto f(x)a).
	\end{aligned}
	\]
	are topological isomorphisms for $n \in \mathbb{N}$. As a corollary, we have
	\[
	\mathscr{S}(\mathbb{R}^m) \hat{\otimes} \mathscr{S}(\mathbb{R}^n) \cong \mathscr{S}(\mathbb{R}^m, \mathscr{S}(\mathbb{R}^n)) \cong \mathscr{S}(\mathbb{R}^{n+m}),
	\]
	\[
	\mathscr{S}(\mathbb{T}^m) \hat{\otimes} \mathscr{S}(\mathbb{T}^n) \cong \mathscr{S}(\mathbb{T}^m, \mathscr{S}(\mathbb{T}^n)) \cong \mathscr{S}(\mathbb{T}^{n+m}),
	\]
\end{proposition}

This proposition gives us another way to differentiate and integrate vector-valued Schwartz functions.

\begin{definition}
	Let $A$ be a Fr\'echet algebra. Then for $G = \mathbb{T}, \mathbb{R}$ we define the derivative \\ $\frac{d}{d x} : \mathscr{S}(G, A) \rightarrow \mathscr{S}(G, A)$ and the integral $\int_G : \mathscr{S}(G, A) \rightarrow A$ using the universal property of the completed projective tensor product:
	\[
	\frac{d}{dx}(f \otimes a) := \left( \frac{d}{dx}f(x) \right) \otimes a , \quad \int_G (f \otimes a) d\mu := \left(  \int_G f d \mu \right)  \otimes a.
	\]
\end{definition}

\begin{definition}
	\label{special actions on Frechet algebras}
	Let $A$ be a Fr\'echet-Arens-Michael algebra, and let $G = \mathbb{R}$ or $G = \mathbb{T}$. Then the action of $G$ on a $A$ is called:
	\begin{enumerate}[label=(\alph*)]
		\item \textit{$m$-tempered}, if there exists a generating family of submultiplicative seminorms $\{ \left\| \cdot \right\|_m\}_{m \in \mathbb{N}}$ on $A$ such that for every $m \in \mathbb{N}$ there is a polynomial $p_m(x) \in \mathbb{R}[x]$, satisfying
		\[
		\left\| \alpha_x(a) \right\|_m \le |p_m(x)| \left\| a \right\|_m \quad (a \in A, x \in G).
		\]
		
		\item \textit{$C^\infty$-$m$-tempered} or \textit{smooth $m$-tempered} , if the following conditions are satisfied:
		\begin{enumerate}[label=(\arabic*)]
			\item for every $a \in A$ the function
			\[
			\alpha_x(a) : G \longrightarrow A, \quad x \mapsto \alpha_x(a),
			\]
			is $C^\infty$-differentiable,
			\item there exists a generating family of submultiplicative seminorms $\{ \left\| \cdot \right\|_m\}_{m \in \mathbb{N}}$ on $A$ such that for any $k \ge 0$ and $m > 0$ there exists a polynomial $p_{k, m} \in\mathbb{R}[x]$, satisfying
			\[
			\left\| \alpha_x^{(k)}(a) \right\|_m \le |p_{k, m}(x)| \left\| a \right\|_m \quad (k \in \mathbb{N}, x \in G, a \in A).
			\]
		\end{enumerate}
	\end{enumerate}
\end{definition}

The following theorem can be viewed as a definition of smooth crossed products.

\begin{theorem}[\cite{schweitzer1993dense}, Theorem 3.1.7]
	\label{def of smooth ore extensions}
	Let $A$ be a Fr\'echet-Arens-Michael algebra with an \newline $m$-tempered action of one of the groups $G = \mathbb{R}$ or $G = \mathbb{T}$. Then the space $\mathscr{S}(G, A)$ endowed with the following multiplication:
	\[
	(f *_\alpha g)(x) = \int_G f(y) \alpha_y(g(x-y)) dy
	\]
	becomes a Fr\'echet-Arens-Michael algebra. 
\end{theorem}
When $G = \mathbb{R}$, we will denote this algebra by $\mathscr{S}(\mathbb{R}, A; \alpha)$, and in the case $G = \mathbb{T}$ we will write $C^\infty(\mathbb{T}, A; \alpha)$.

\textbf{Remark.} If $\alpha$ is the trivial action, then $\mathscr{S}(G, A; \alpha) = \mathscr{S}(G, A)$ with the usual convolution product.

\begin{proposition}
	\label{criteria for smooth act}
	Let $A$ be a Fr\'echet-Arens-Michael algebra. Consider an action $\alpha : \mathbb{R} \rightarrow \text{Aut}(A)$. Then $\alpha$ is a smooth $m$-tempered action if and only if the following holds:
	\begin{enumerate}
		\item the derivative $\alpha'_x(a)$ exists at $x = 0$ for every $a \in A$, and, as a corollary, derivatives all of orders at zero exist.
		\item there exists a generating family of submultiplicative seminorms $\{ \left\| \cdot \right\|_m\}_{m \in \mathbb{N}}$ on $A$ such that for every $m \in \mathbb{N}$ and $k \in \mathbb{N}$ there exist polynomials $p_m(x) \in \mathbb{R}[x]$ and $C_{k,m} > 0$, satisfying
		\[
		\left\| \alpha_x(a) \right\|_m \le |p_m (x)| \left\| a \right\|_m, \quad \left\| \alpha_0^{(k)}(a) \right\|_m \le C_{k, m} \left\| a \right\|_m  
		\]
		for every $a \in A$, $x \in \mathbb{R}$.
	\end{enumerate}
\end{proposition}

\begin{proof}
	$(\Rightarrow)$ If $\alpha$ is $C^\infty$-$m$-tempered, choose the seminorms $\left\| \cdot \right\|_m$ and the polynomials $p_{m, k}(x)$ as in the Definition \ref{special actions on Frechet algebras}, and set
	\[
	p_m(x) = p_{0, m}(x), \quad C_{k, m} = p_{k, m}(0).
	\]
	
	$(\Leftarrow)$ Notice that
	\begin{equation}
	\label{derivation}
		\alpha'_x(a) = \lim_{h \rightarrow 0} \frac{\alpha_{x+h}(a) - \alpha_x(a)}{ h } = \alpha_x \left( \lim_{h \rightarrow 0} \frac{\alpha_{h}(a) - a}{h} \right) = \alpha_x (\alpha'_0(a)) \quad (a \in A).
	\end{equation}
	Therefore,
	\[
	\alpha^{(k)}_x(a) = \alpha^{(k-1)}_x(\alpha'_0(a)) = \alpha^{(k-2)}_x(\alpha'_0(\alpha'_0(a))) =  \dots = \alpha_x(\underbrace{\alpha'_0(\dots (\alpha'_0(a)))}_{k \text{ times}}).
	\]
	However,
	\[
	\alpha'_0(\alpha'_0(x)) = \lim_{h \rightarrow 0} \frac{\alpha_h(\alpha'_0(x)) - \alpha'_0(x)}{h} = \lim_{h \rightarrow 0} \frac{\alpha'_h(x) - \alpha'_0(x)}{h} = \alpha''_0(x).
	\]
	By induction we obtain the following equality:
	\begin{equation}
	\label{derivative of alpha}
		\alpha^{(k)}_x(a) = \alpha_x(\alpha^{(k)}_0(a))
	\end{equation}
	for every $a \in A$, $x \in \mathbb{R}$, $k \in \mathbb{Z}_{\ge 0}$.

	As an immediate corollary, $\alpha_x(a) \in C^\infty(\mathbb{R}, A)$ for every $a \in A$. This also implies that
	\[
	\left\| \alpha^{(k)}_x(a) \right\|_m = \left\| \alpha_x( \alpha^{(k)}_0(x) ) \right\|_m \le \left| p_m(x) \right| \left\| \alpha_0^{(k)}(a) \right\|_m \le \left| p_m(x) \right| C_{k, m} \left\| a \right\|_m.
	\] 
	Now set $p_{k, m}(x) = C_{k, m} p_m(x)$.
\end{proof}

The proposition can be restated for $G = \mathbb{T}$:

\begin{proposition}
	Let $A$ be a Fr\'echet-Arens-Michael algebra. Consider an action $\alpha : \mathbb{T} \rightarrow \text{Aut}(A)$. Then $\alpha$ is a smooth $m$-tempered action if and only if the following holds:
	\begin{enumerate}
		\item the derivative $\alpha'_x(a)$ exists at $x = 0$ for every $a \in A$, and, as a corollary, derivatives all of orders at zero exist.
		\item there exists a generating family of submultiplicative seminorms $\{ \left\| \cdot \right\|_m\}_{m \in \mathbb{N}}$ on $A$ such that for every $m \in \mathbb{N}$ and $k \in \mathbb{N}$ there exist $C_m, C_{k,m} > 0$, satisfying
		\[
		\left\| \alpha_x(a) \right\|_m \le C_m \left\| a \right\|_m, \quad \left\| \alpha_0^{(k)}(a) \right\|_m \le C_{k, m} \left\| a \right\|_m  
		\]
		for every $a \in A$, $x \in \mathbb{T}$.
	\end{enumerate}
\end{proposition}
\begin{proof}
	The proof is the same as in the previous proposition, we only need keep in mind that 
	\[
	|p_{k, m}(x)| \le \sup_{x \in \mathbb{T}} |p_{k, m}(x)| < \infty. 
	\]
\end{proof}
\subsection{Explicit construction}

\textbf{Remark.} In this subsection we only treat the case $G = \mathbb{R}$ here, the case $G = \mathbb{T}$ can be dealt with in the same way.

\begin{definition}
	\label{self-induced}
	A $\hat{\otimes}$-algebra $A$ is called \textit{self-induced}, if the multiplication map $m : A \hat{\otimes}_A A \longrightarrow A$ is a $A$-$\hat{\otimes}$-bimodule isomorphism.
\end{definition}
Until the end of this section, $A$ will denote a self-induced Fr\'echet-Arens-Michael algebra. Also let $\alpha$ denote a smooth $m$-tempered $\mathbb{R}$-action of $A$.

In this subsection we will construct a $\hat{\Omega}^1_A$-like admissible sequence for $\mathscr{S}(\mathbb{R}, A ;\alpha)$.

\begin{proposition}
	\label{properties of the operator T}
	For any $F \in \mathscr{S}(\mathbb{R}, A)$ define $T(F)(x) = \alpha_x(F(x))$. Then the following statements hold:
	\begin{enumerate}
		\item The mapping $T$ is a well-defined continuous linear map $T : \mathscr{S}(\mathbb{R}, A) \rightarrow \mathscr{S}(\mathbb{R}, A)$,
		\item Moreover, $T$ is invertible, with the inverse, defined for every $F \in \mathscr{S}(\mathbb{R}, A)$ as follows:
		\[
		T^{-1}(F)(x) = \alpha_{-x}(F(x)).
		\]
		In particular, we have
		\begin{equation}
		\label{eq4}
			\left( T \circ \frac{\text{d}}{\text{d}x} \circ T^{-1}\right)(F)(x) = F'(x) -\alpha'_0(F(x)) 
		\end{equation}
		for any $F \in \mathscr{S}(\mathbb{R}, A)$.
		\item For any $F, G \in \mathscr{S}(\mathbb{R}, A ;\alpha)$ we have $$F' *_\alpha T(G) = F *_\alpha T(G').$$ This equality is equivalent to $$F' *_\alpha G = F *_\alpha \left( T \circ \frac{\text{d}}{\text{d}x} \circ T^{-1}\right)(G).$$
	\end{enumerate}
\end{proposition}
\begin{proof}
	\indent
	\begin{enumerate}
		\item Let us write down the derivative of $\alpha_x(F(x))$:
		\[
		\begin{aligned}
		\frac{\text{d}}{\text{d} x} (\alpha_x(F(x))) & = \lim_{h \rightarrow 0}  \frac{\alpha_{x+h}(F(x+h)) - \alpha_x(F(x))}{h} = \\ &= \alpha_x\left( \lim_{h \rightarrow 0} \frac{\alpha_h(F(x) + F'(x)h + o(h)) - F(x)}{h} \right) = \\
		&= \alpha_x \left( \alpha'_0(F(x)) + F'(x)\right) = \alpha'_x(F(x)) + \alpha_x(F'(x)). 
		\end{aligned}
		\]
		It is easily seen that
		\begin{equation}
			\frac{\text{d}^k}{\text{d} x^k} (\alpha_x(F(x))) = \sum_{i = 0}^k \binom{n}{k} \alpha^{(i)}_x(F^{(k-i)}(x)).
		\end{equation}
		Now fix a generating system of seminorms on $A$ which satisfies the conditions of the Proposition \ref{criteria for smooth act}. Let us show that $\alpha^{(m)}_x(F(x))$ lies in $\mathscr{S}(\mathbb{R}, A)$ for any $F \in \mathscr{S}(\mathbb{R}, A)$ and $m \ge 0$:
		\[
		\begin{aligned}
		&\left\| T(F) \right\|_{k, l, m} =  \sup_{x \in \mathbb{R}}\left\| x^l \frac{\text{d}^k}{\text{d} x^k} (\alpha_x(F(x))) \right\|_m \le \sum_{i = 0}^k \binom{k}{i} \sup_{x \in \mathbb{R}}\left\| \alpha^{(i)}_x(F^{(k-i)}(x))  \right\| \le \\ &\le \sum_{i = 0}^k \binom{k}{i} \sup_{x \in \mathbb{R}} \left(  |p_{i, m}(x)| \left\| F^{(k-i)}(x) \right\|_m \right) < \infty.
		\end{aligned}
		\]
		
		\item Notice that the same argument shows works for $T^{-1}$, as well. As for the equality, notice that 
		\begin{equation}
		\label{eq5}
			\frac{\text{d}}{\text{d} x} (\alpha_{-x}(F(x))) = -\alpha'_{-x}(F(x)) + \alpha_{-x}(F'(x)),
		\end{equation}
		so we have
		\[
		T\left( \frac{\text{d}}{\text{d} x} (\alpha_{-x}(F(x)))\right) = -\alpha_x(\alpha'_0(F(x))) + (F'(x)) \stackrel{\ref{derivative of alpha}}{=} -\alpha'_0(F(x)) + F'(x)
		\]
		\item This is equivalent to
		\[
		\begin{aligned}
		\int_{\mathbb{R}} F'(y) \alpha_y(TG(x-y)) \text{d} y &= \int_{\mathbb{R}} F(y) \alpha_y(TG'(x-y)) \text{d} y \Leftrightarrow \\
		\Leftrightarrow \int_{\mathbb{R}} F'(y) \alpha_x(G(x-y)) \text{d} y &= \int_{\mathbb{R}} F(y) \alpha_x(G'(x-y)) \text{d} y \Leftrightarrow \\
		\Leftrightarrow \int_{\mathbb{R}} \frac{\text{d}}{\text{d}y} (F(y)) \alpha_x(G(x-y)) \text{d} y &= - \int_{\mathbb{R}} F(y) \frac{\text{d}}{\text{d}y}(\alpha_x(G(x-y))) \text{d} y \Leftrightarrow\\
		\Leftrightarrow \int_{\mathbb{R}} \frac{\text{d}}{\text{d}y} (F(y)) \alpha_x(G(x-y)) \text{d} y &+ \int_{\mathbb{R}} F(y) \frac{\text{d}}{\text{d}y}(\alpha_x(G(x-y))) \text{d} y = 0 \quad \text{(integration by parts)}
		\end{aligned}
		\]
	\end{enumerate}
\end{proof}

Let $\mathscr{S}(\mathbb{R}, A; \alpha)_\alpha =: S_\alpha$ denote the $\mathscr{S}(\mathbb{R}, A; \alpha)$-$\hat{\otimes}$-bimodule and $A$-$\hat{\otimes}$-bimodule, which coincides with $\mathscr{S}(\mathbb{R}, A)$ as a LCS, and the bimodule actions are given below:
\[
\begin{aligned}
(F \circ a)(x) &= F(x) \alpha_x(a), & \quad a \circ F(x) &= a F(x) \quad & \text{ for any } &a \in A, F \in S_\alpha \\
(F \circ G)(x) &= (F *_\alpha G)(x), & \quad (G \circ F)(x) &= (G *_\alpha F)(x) \quad & \text{ for any } &F, G \in S_\alpha
\end{aligned}
\]

\begin{proposition}
	\indent
	\begin{enumerate}
		\item For any $F \in \mathscr{S}(\mathbb{R}, A)$ and $a \in A$ the functions $F \circ a$ and $a \circ F$ belong to $\mathscr{S}(\mathbb{R}, A)$. As a corollary, $S_\alpha$ is well-defined.
		\item The following equalities take place:
		\begin{equation}
		\label{eq1}
			(F \circ a)' = (F' \circ a) + (F \circ \alpha'_0(a)),
		\end{equation}
		\begin{equation}
		\label{eq2}
			\left( T \circ \frac{\text{d}}{\text{d}x} \circ T^{-1}\right) (F \circ a) = \left( \left( T \circ \frac{\text{d}}{\text{d}x} \circ T^{-1}\right) F\right)  \circ a,
		\end{equation}
		\begin{equation}
		\label{eq3}
			\left( T \circ \frac{\text{d}}{\text{d}x} \circ T^{-1}\right) (a \circ F) = a \circ \left( T \circ \frac{\text{d}}{\text{d}x} \circ T^{-1}\right)(F) - \alpha'_0(a) \circ F.
		\end{equation}
	\end{enumerate}
\end{proposition}
\begin{proof}
	\indent
	\begin{enumerate}
		\item 
		The argument for $a \circ F$ is pretty much trivial, we only need to check that $F \circ a \in \mathscr{S}(\mathbb{R}, A)$. Fix a generating system of seminorms $\{ \left\| \cdot \right\|_m \}$ on $A$, satisfying the conditions of the Proposition \ref{criteria for smooth act}.
		
		Notice that for every $k, l \in \mathbb{Z}_{\ge 0}$ and $m \in \mathbb{N}$ we have
		\[
		\begin{aligned}
		\left\| F \circ a \right\|_{k, l, m} &=  \sup_{x \in \mathbb{R}} \left\| x^l \frac{\text{d}^k}{\text{d} x^k} (F \circ a)(x) \right\|_m = \sup_{x \in \mathbb{R}} \left\| \sum_{i = 0}^k x^l  \binom{k}{i} F^{(i)}(x) \alpha^{(k-i)}_0(a) \right\|_m \le \\
		&\le \sum_{i = 0}^k \binom{k}{i} |x^l| \left\| F^{(i)}(x) \right\|_m \left\| \alpha_0^{(k-i)}(a) \right\|_m \le \left\| a \right\|_m  \sum_{i = 0}^k \binom{k}{i} C_{k-i, m} |x^l| \left\| F^{(i)}(x) \right\|_m = \\
		&= \left\| a \right\|_m  \sum_{i = 0}^k \binom{k}{i} C_{k-i, m} \left\| F \right\|_{i, l, m} < \infty.
		\end{aligned}
		\]
		
		\item
		Checking these equalities is pretty straightforward:
		\[
		(F \circ a)'(x) = F(x) \alpha'_x(a) + F'(x) \alpha_x(a) \stackrel{\ref{derivative of alpha}}{=} F(x) \alpha_x(\alpha'_0(a)) + F'(x) \alpha_x(a) = (F' \circ a)(x) + (F \circ \alpha'_0(a))(x),
		\]
		\[
		\left( T \circ \frac{\text{d}}{\text{d}x} \circ T^{-1}\right) (F \circ a) = \left(  \left( T \circ \frac{\text{d}}{\text{d}x} \right) (T^{-1}F) \circ a\right) = \left( \left( T \circ \frac{\text{d}}{\text{d}x} \circ T^{-1}\right) F\right)  \circ a.
		\]
		\[
		\begin{aligned}
			&\left( T \circ \frac{\text{d}}{\text{d}x} \circ T^{-1}\right) (a \circ F)(x) = \left( T \circ \frac{\text{d}}{\text{d}x} \right) (\alpha_{-x}(a) (T^{-1}F)(x)) = \\ &= T\left( \alpha_{-x}(a) (T^{-1}F)'(x) - \alpha'_{-x}(a) (T^{-1}F)(x) \right) 
			= \left( a \circ \left( T \circ \frac{\text{d}}{\text{d}x} \circ T^{-1}\right)(F) - \alpha'_0(a) \circ F\right)(x).
		\end{aligned}
		\]
	\end{enumerate}
\end{proof}

\begin{lemma}
	\label{lemma2.1}
	The bimodule $S_\alpha$ belongs to the categories $\mathscr{S}(\mathbb{R}, A; \alpha)$-\textbf{mod}-$A(\textbf{Fr})$ and $A$-\textbf{mod}-$\mathscr{S}(\mathbb{R}, A; \alpha)(\textbf{Fr})$.
	In particular, the following $\mathscr{S}(\mathbb{R}, A; \alpha)$-$\hat{\otimes}$-bimodule structure on $S_\alpha \hat{\otimes}_A S_\alpha$ is well-defined:
	\[
	H \circ (F \otimes G) = (H *_\alpha F) \otimes G, \quad (F \otimes G) \circ H = F \otimes (G *_\alpha H)
	\]
	for any $F, G, H \in \mathscr{S}(\mathbb{R}, A)$.
\end{lemma}
\begin{proof}
	We only need to prove that $(H \circ F) \circ a = H \circ (F \circ a)$ and $(a \circ F) \circ H = a \circ (F \circ H)$ for any $a \in A; F, H \in S_\alpha$.
	\[
	H \circ (F \circ a)(x) = \int_\mathbb{R} H(y) \alpha_y(F(x-y)) \alpha_x(a) \text{d}y = (H \circ F) \circ a(x),
	\]
	\[
	(a \circ F) \circ H(x) = \int_\mathbb{R} aF(y) \alpha_y(H(x-y)) \text{d}y = a \circ (H \circ F)(x).
	\]
\end{proof}

It is also easy to see that $S_\alpha \hat{\otimes}_A S_\alpha$ is a well-defined $A$-$\hat{\otimes}$-bimodule.

\begin{lemma}
	\label{lemma2.2}
	Define the following maps:
	\[
	\begin{aligned}
	m : S_\alpha \hat{\otimes}_A S_\alpha & \longrightarrow S_\alpha, & \quad m(F \otimes G) &= F *_\alpha G, \\
	j : S_\alpha \hat{\otimes}_A S_\alpha & \longrightarrow S_\alpha \hat{\otimes}_A S_\alpha, & \quad j(F \otimes G) &= F' \otimes G - F \otimes \left( T \circ \frac{\text{d}}{\text{d}x} \circ T^{-1}\right) G .
	\end{aligned}
	\]
	These maps are well-defined $\mathscr{S}(\mathbb{R}, A; \alpha)$-$\hat{\otimes}$-bimodule and $A$-$\hat{\otimes}$-bimodule homomorphisms.
\end{lemma}

\begin{proof}
	First of all, let us prove that $m$ and $j$ are well-defined:
	\[
	m(F \circ a \otimes G)(x) = \int_\mathbb{R} F(y)\alpha_y(a) \alpha_y(G(x-y)) \text{d} y = \int_\mathbb{R} F(y) \alpha_y(aG(x-y)) \text{d} y = m(F \otimes a \circ G)(x, y).
	\]
	\[
	\begin{split}
	j(F \circ a \otimes G) &= (F \circ a)' \otimes G - (F \circ a) \otimes \left( T \circ \frac{\text{d}}{\text{d}x} \circ T^{-1}\right)(G) \stackrel{\ref{eq1}}{=} \\
	&\stackrel{\ref{eq1}}{=} (F' \circ a) \otimes G + (F \circ \alpha'_0(a)) \otimes G - (F \circ a) \otimes \left( T \circ \frac{\text{d}}{\text{d}x} \circ T^{-1}\right)(G) = \\
	&= F' \otimes a \circ G + F \otimes \alpha'_0(a) \circ G - F \otimes a \circ \left( T \circ \frac{\text{d}}{\text{d}x} \circ T^{-1}\right)(G).
	\end{split}
	\]
	\[
	\begin{split}
	j(F \otimes a \circ G) &= F' \otimes a \circ G - F(x) \otimes \left( T \circ \frac{\text{d}}{\text{d}x} \circ T^{-1}\right) (a \circ G) \stackrel{\ref{eq3}}{=} \\
	&\stackrel{\ref{eq3}}{=} F' \otimes a \circ G - F \otimes a \circ \left( T \circ \frac{\text{d}}{\text{d}x} \circ T^{-1}\right)(G) + F \otimes \alpha'_0(a) \circ G = j(F \circ a \otimes G).
	\end{split}
	\]
	The algebra $\mathscr{S}(\mathbb{R}, A; \alpha)$ is associative, therefore, $m$ is a $\mathscr{S}(\mathbb{R}, A; \alpha)$-$\hat{\otimes}$-bimodule homomorphism.
	
	It is relatively easy to show that $j$ is a left $\mathscr{S}(\mathbb{R}, A; \alpha)$-$\hat{\otimes}$-module homomorphism:
	\[
	\begin{split}
	j((H *_\alpha F) \otimes G) &= (H *_\alpha F)' \otimes G - (H *_\alpha F) \otimes \left( T \circ \frac{\text{d}}{\text{d}x} \circ T^{-1}\right)(G) = \\
	&= (H *_\alpha F') \otimes G - (H *_\alpha F) \otimes \left( T \circ \frac{\text{d}}{\text{d}x} \circ T^{-1}\right)(G) = H *_\alpha j(F \otimes G).
	\end{split}
	\]
	And it is slightly more difficult to show that it is a right $\mathscr{S}(\mathbb{R}, A; \alpha)$-$\hat{\otimes}$-module homomorphism.
	\[
	\begin{split}
	j( F \otimes (G *_\alpha H)) &= F' \otimes (G *_\alpha H) - F \otimes T((T^{-1}(G * H))')
	\end{split}
	\]
	\[ 
	T^{-1}(G * H)(x) = \int_\mathbb{R} \alpha_{-x}(G(y)) \alpha_{-x+y}(H(x-y)) \text{d}y = \int_\mathbb{R} \alpha_{-x}(G(y+x)) \alpha_y(H(-y)) \text{d}y
	\]
	\[
	\begin{aligned}
		(T^{-1}(G * H))'(x) &=  \int_{\mathbb{R}} \frac{\text{d}}{\text{d}x} \left( \alpha_{-x}(G(y+x)) \right)  \alpha_y(H(-y)) \text{d}y \stackrel{\ref{eq5}}{=} \\
		&\stackrel{\ref{eq5}}{=} \int_{\mathbb{R}} (\alpha_{-x}(G'(x+y)) - \alpha'_{-x}(G(x+y)) ) \alpha_y(H(-y)) \text{d}y.
	\end{aligned}
	\]
	\[
	\begin{split}
	T((T^{-1}(G * H))')(x) &= \int_{\mathbb{R}} ( G'(x+y) - \alpha_x(\alpha'_{-x}(G(x+y)))) \alpha_{x+y}(H(-y)) \text{d}y \stackrel{\ref{derivative of alpha}}{=} \\ &\stackrel{\ref{derivative of alpha}}{=} \int_{\mathbb{R}} (G'(x+y) - \alpha'_0(G(x+y))) \alpha_{x+y}(H(-y)) \text{d}y = \\
	&= \int_{\mathbb{R}} (G'(y) - \alpha'_0(G(y))) \alpha_{y}(H(x-y)) \text{d}y \stackrel{\ref{eq4}}{=} \left(  \left( T \circ \frac{\text{d}}{\text{d}x} \circ T^{-1}\right)(G) *_\alpha H \right) (x).
	\end{split}
	\]
	Therefore, we have
	\[
	j( F \otimes (G *_\alpha H)) = F' \otimes (G *_\alpha H) - F \otimes \left( T \circ \frac{\text{d}}{\text{d}x} \circ T^{-1}\right)(G) *_\alpha H = j(F \otimes G) *_\alpha H.
	\]
	Now let us check that $j$ and $m$ are $A$-$\hat{\otimes}$-bimodule homomorphisms:
	\[
	\begin{aligned}
	m(a \circ F \otimes G)(x) = \int_R a F(y) \alpha_y(G(x-y)) \text{d}y = (a \circ m(F \otimes G))(x) \\
	m(F \otimes G \circ a)(x) = \int_R F(y) \alpha_y(G(x-y)) \alpha_x(a) \text{d}y = (m(F \otimes G) \circ a)(x) 
	\end{aligned}
	\]
	\[
	\begin{aligned}
	j(a \circ F \otimes G) &= (a \circ F)' \otimes G - a \circ F \otimes \left( T \circ \frac{\text{d}}{\text{d}x} \circ T^{-1}\right)(G) = a \circ j(F \otimes G) \\
	j(F \otimes G \circ a) &= F' \otimes G \circ a - F \otimes \left( T \circ \frac{\text{d}}{\text{d}x} \circ T^{-1}\right)(G \circ a) \stackrel{\ref{eq2}}{=}\\
	&\stackrel{\ref{eq2}}{=} F' \otimes G \circ a - F(x) \otimes \left( T \circ \frac{\text{d}}{\text{d}x} \circ T^{-1}\right)(G) \circ a = j(F \otimes G) \circ a.
	\end{aligned}
	\]
\end{proof}
As a corollary from the Proposition \ref{properties of the operator T} we have $m \circ j = 0$.

\begin{proposition}
	The tensor product $S_\alpha \hat{\otimes}_A S_\alpha$ is isomorphic to $\mathscr{S}(\mathbb{R}^2, A)$ as a locally convex space:
	\[
	\begin{aligned}
	I_1 : S_\alpha \hat{\otimes}_A S_\alpha \longrightarrow \mathscr{S}(\mathbb{R}^2, A), \quad I_1(F \otimes G)(x, y) = \alpha_{-x}(F(x))G(y). \\
	\end{aligned}
	\]
\end{proposition}
\begin{proof}
	First of all, we can replace $\mathscr{S}(\mathbb{R}^2, A)$ with $\mathscr{S}(\mathbb{R}^2, A \hat{\otimes}_A A)$, because $A$ is isomorphic to $A \hat{\otimes}_A A$ as a locally convex space. This is precisely where we use the fact that $A$ is a self-induced algebra.
	
	Let $X$ be a complete LCS and consider a continuous $A$-balanced map $Q : S_\alpha \times S_\alpha \longrightarrow X$. Define the map
	\[
	\widetilde{Q} : \mathscr{S}(\mathbb{R}^2, A \hat{\otimes}_A A) \longrightarrow X, \quad \widetilde{Q}(f(x)g(y) a \otimes b) = I_1( f(x)\alpha_x(a), g(x)b).
	\]
	This map is a well-defined continuous linear map, because $I_1$ is $A$-balanced and the linear span of $\{ f(x)g(y)a \otimes b : f, g \in \mathscr{S}(\mathbb{R}), a, b \in A\}$ is dense in $\mathscr{S}(\mathbb{R}^2, A \hat{\otimes}_A A)$. From the construction of $\widetilde{Q}$ it follows that the following diagram is commutative:
	\[
	\begin{tikzcd}
		\mathscr{S}(\mathbb{R}^2, A \hat{\otimes}_A A) \ar[r, "\widetilde{Q}"] & X \\
		S_\alpha \times S_\alpha \ar[u, "I_1"] \ar[ur, "Q"']
	\end{tikzcd}
	\]
	Moreover, $\widetilde{Q}$ is a unique mapping which makes this diagram commute.
\end{proof}

Therefore, the isomorphism $I_1$ induces the structure of $S_\alpha$-module on $\mathscr{S}(\mathbb{R}^2, A)$, which we will denote by $\mathscr{S}(\mathbb{R}^2, A)_\alpha$. Let us describe the action of $\mathscr{S}(\mathbb{R}, A; \alpha)$ and $A$ on $\mathscr{S}(\mathbb{R}^2, A)_\alpha$ explicitly.

\begin{lemma}
	\label{lemma2.3}
	The algebra $A$ acts on $\mathscr{S}(\mathbb{R}^2, A)_\alpha$ as follows:
	\[
	\begin{aligned}
		(a \circ F)(x, y) &= \alpha_{-x}(a)F(x, y) & \quad \text{ for any } &F \in \mathscr{S}(\mathbb{R}^2, A)_\alpha, a \in A  \\
		(F \circ a)(x, y) &= F(x, y)\alpha_y(a) & \quad \text{ for any } &F \in \mathscr{S}(\mathbb{R}^2, A)_\alpha, a \in A
	\end{aligned}
	\]
	The algebra $\mathscr{S}(\mathbb{R}, A; \alpha)$ acts on $\mathscr{S}(\mathbb{R}^2, A)_\alpha$ as follows:
	\[
	\begin{aligned}
	(H \circ F)(x, y) &= \int_{\mathbb{R}} \alpha_{-x}(H(z)) F(x - z, y) \text{d}z & \quad \text{ for any } &F \in \mathscr{S}(\mathbb{R}^2, A)_\alpha, H \in \mathscr{S}(\mathbb{R}, A; \alpha) \\
	(F \circ H)(x, y) &= \int_{\mathbb{R}} F(x, z) \alpha_z(H(y-z)) \text{d} z & \quad \text{ for any } &F \in \mathscr{S}(\mathbb{R}^2, A)_\alpha, H \in \mathscr{S}(\mathbb{R}, A; \alpha).
	\end{aligned}
	\]
\end{lemma}
\begin{proof}
	In all cases we will check every relation on a dense subset, then we will use the continuity arguments to finish the proof. Let $F(x, y) = G_1(x) G_2(y)$.
	\[
	\begin{aligned}
	(a \circ F)(x, y) &= I_1( a \circ TG_1 \otimes G_2)(x, y) = \alpha_{-x}(a) G_1(x) G_2(y) = \alpha_{-x}(a)F(x, y) \\
	(F \circ a)(x, y) &= I_1(TG_1 \otimes G_2 \circ a)(x, y) = G_1(x)G_2(y)\alpha_y(a) = F(x, y) \alpha_y(a)
	\end{aligned}
	\]
	\[
	\begin{aligned}
	&(H \circ F)(x, y) = I_1\left( \int_{\mathbb{R}}H(z)\alpha_z(TG_1(x-z)) \text{d}z \otimes G_2(x) \right) \otimes G_2(x) = \\ &= I_1 \left( \int_{\mathbb{R}}H(z)\alpha_x(G_1(x-z)) \text{d}z \otimes G_2(x) \right) = \left( \int_{\mathbb{R}}\alpha_{-x}(H(z)) (G_1(x-z)) G_2(y) \text{d}z \right) =\\
	&= \int_{\mathbb{R}} \alpha_{-x}(H(z)) F(x - z, y) \text{d}z.
	\end{aligned}
	\]
	\[
	\begin{aligned}
	&(F \circ H)(x, y) = I_1 \left( TG_1(x) \otimes \int_{\mathbb{R}} G_2(z) \alpha_z(H(x-z)) \text{d} z \right) =  \int_{\mathbb{R}} G_1(x) G_2(z) \alpha_z(H(y -z)) \text{d} z = \\
	&= \int_{\mathbb{R}} F(x, z) \alpha_z(H(y-z)) \text{d} z.
	\end{aligned}
	\]
\end{proof}

Now we want to construct the right inverse maps to $m$ and $j$. First of all, let us describe the action of these maps on $\mathscr{S}(\mathbb{R}^2, A)$.

\begin{lemma}
	\label{lemma2.4}
	The following diagrams are commutative:
	\[
	\begin{tikzcd}
	S_\alpha \hat{\otimes}_A S_\alpha \ar[d, "I_1"] \ar[r, "j"] & S_\alpha \hat{\otimes}_A S_\alpha \ar[d, "I_1"] & S_\alpha \hat{\otimes}_A S_\alpha \ar[d, "I_1"] \ar[r, "m"] & S,  \\
	\mathscr{S}(\mathbb{R}^2, A)_\alpha \ar[r, "\iota"] & \mathscr{S}(\mathbb{R}^2, A)_\alpha  & \mathscr{S}(\mathbb{R}^2, A)_\alpha \ar[ru, "\pi"] & 
	\end{tikzcd}
	\]
	where 
	\[
	\begin{aligned}
	\iota(F)(x, y) &= \left( \frac{\partial F }{\partial x} - \frac{\partial F}{\partial y}  \right)(x, y) + \alpha'_0(F(x, y))  & \text{ for any } F\in \mathscr{S}(\mathbb{R}^2, A) \\
	\pi(F)(x) &= \int_{\mathbb{R}} \alpha_y ( F(y,x-y) ) \text{d} y & \text{ for any } F \in \mathscr{S}(\mathbb{R}^2, A)
	\end{aligned}
	\]
\end{lemma}
\begin{proof}
	It is obvious that $\iota$ and $\rho$ are continuous, so we can assume that $F(x, y) = G(x)H(y)$ for some $G, H \in S$:
	\[
	\begin{aligned}
		&I_1 \circ j \circ I^{-1}_1(F)(x, y) = I_1 \circ j(TG \otimes H)(x, y) = I_1\left( (TG)'(x)H(y) - (TG)(x)(T(T^{-1}H)(y)) \right)  = \\ &= T^{-1}((TG)'(x)) H(y) - G(x) T((T^{-1}H)'(y)) \stackrel{\ref{eq4}}{=} (G'(x) + \alpha'_0(G(x))) H(y) - G(x)(H'(y) - \alpha'_0(H(y))) = \\
		&= \left( \frac{\partial F }{\partial x} - \frac{\partial F}{\partial y}  \right)(x, y) + \alpha'_0(G(x))H(y) + G(x)\alpha'_0(H(y)) =  \left( \frac{\partial F }{\partial x} - \frac{\partial F}{\partial y}  \right)(x, y) + \frac{\text{d}}{\text{d} t} \left. \alpha_t(G(x)H(y))\right|_{t=0}.
	\end{aligned}
	\]
	\[
	m \circ I_1^{-1}(F)(x) = \int_R TG(y) \alpha_y(H(x-y)) \text{d}y = \int_R \alpha_y(G(y)H(x-y)) \text{d}y = \pi(F)(x).
	\]
\end{proof}

Now we can construct the right inverse maps for $\rho$.

\begin{lemma}
	\label{lemma2.5}
	Fix a function $\varphi \in C^\infty_c(\mathbb{R})$ with $\int_{\mathbb{R}} \varphi(t) \text{d}t = 1$. Define the maps
	\[
	\begin{aligned}
	\rho_x : S_\alpha & \longrightarrow \mathscr{S}(\mathbb{R}^2, A)_\alpha, \quad 
	\rho_x(F)(x, y) = \varphi(y)\alpha_{-x}(F(x+y)), \\
	\rho_y : S_\alpha & \longrightarrow \mathscr{S}(\mathbb{R}^2, A)_\alpha, \quad 
	\rho_y(F)(x, y) = \varphi(x) \alpha_{-x}(F(x+y)).
	\end{aligned}
	\]
	Then $\rho_x$ is a $\mathscr{S}(\mathbb{R}, A; \alpha)$-$A$-$\hat{\otimes}$-bimodule homomorphism, and $\rho_y$ is a $A$-$\mathscr{S}(\mathbb{R}, A; \alpha)$-$\hat{\otimes}$-bimodule homomorphism. Moreover, we have
	\[
	\pi \circ \rho_x = \pi \circ \rho_y = \text{Id}_{S_\alpha}.
	\]
	As a corollary, the algebra $\mathscr{S}(\mathbb{R}, A; \alpha)$ is projective as a left and right $\mathscr{S}(\mathbb{R}, A; \alpha)$-$\hat{\otimes}$ module.
\end{lemma}
\begin{proof}
	For any $F, H \in \mathscr{S}(\mathbb{R}, A; \alpha), a \in A$ we have
	\[
	\begin{aligned}
	\rho_x(H *_\alpha F)(x, y) &= \varphi(y) \int_{\mathbb{R}} \alpha_{-x}(H(z)) \alpha_{z-x}(F(x+y-z)) \text{d} z \\
	H \circ \rho_x(F)(x, y) &= \int_{\mathbb{R}} \alpha_{-x}(H(z)) \rho_x(F)(x-z, y) \text{d} z = \int_{\mathbb{R}} \varphi(y) \alpha_{-x}(H(z)) \alpha_{z-x}(F(x+y-z)) \text{d} z \\
	\rho_x(F \circ a)(x, y) &= \varphi(y) \alpha_{-x}((F \circ a)(x+y)) = \varphi(y) F(x+y) \alpha_y(a) = (\rho_x(F)) \circ a \\
	(\pi \circ \rho_x)(F)(x) &= \int_{\mathbb{R}} \alpha_y(\rho_x(F)(y,x-y)) \text{d}y = \int_{\mathbb{R}} \varphi(x - y) F(x)\text{d}y = F(x) \int_{\mathbb{R}} \varphi(x-y) \text{d} y = F(x).
	\end{aligned}
	\]
	\[
	\begin{aligned}
	\rho_y(F *_\alpha H)(x, y) &= \varphi(x) \int_{\mathbb{R}} \alpha_{-x}(F(z)) \alpha_{z-x}(H(x+y-z)) \text{d}z = \int_{\mathbb{R}} \varphi(x) \alpha_{-x}(F(z+x)) \alpha_z(H(y-z)) \text{d}z \\
	\rho_y(F) \circ H(x, y) &= \int_{\mathbb{R}} \rho_y(F)(x, z) \alpha_z(H(y-z)) \text{d}z = \int_{\mathbb{R}} \varphi(x) \alpha_{-x}(F(x+z)) \alpha_z(H(y-z)) \text{d}z \\
	\rho_y(a \circ F)(x, y) &= \varphi(x) \alpha_{-x}((a \circ F)(x + y)) = \varphi(x) \alpha_{-x}(a) F(x+y) = (a \circ \rho_y(F))(x, y) \\
	(\pi \circ \rho_y)(F)(x) &= \int_{\mathbb{R}} \alpha_y(\rho_y(F)(y, x-y)) \text{d}y = \int_{\mathbb{R}} \varphi(y) F(x) \text{d}y = F(x).
	\end{aligned}
	\]
\end{proof}

\begin{lemma}
	\label{lemma2.6}
	Fix a function $\varphi \in C^\infty_c(\mathbb{R})$ with $\int_{\mathbb{R}} \varphi(t) \text{d}t = 1$. Define the maps
	\[
	\begin{aligned}
	\beta_x & : \mathscr{S}(\mathbb{R}^2, A)_\alpha \longrightarrow \mathscr{S}(\mathbb{R}^2, A)_\alpha, \\
	\beta_y & : \mathscr{S}(\mathbb{R}^2, A)_\alpha \longrightarrow \mathscr{S}(\mathbb{R}^2, A)_\alpha, \\
	\beta_x &(F)(x, y) = \int_{-\infty}^x \left( \alpha_{t-x}(F(t, x+y - t)) - \varphi(x+y-t) \int_{\mathbb{R}} \alpha_{z-x}(F(z, x+y-z)) \text{d}z \right) \text{d}t \\
	\beta_y &(F)(x, y) = \int_{-\infty}^x \left( \alpha_{t-x}(F(t, x+y - t)) - \varphi(t) \int_{\mathbb{R}} \alpha_{z-x}(F(z, x+y-z)) \text{d}z \right) \text{d}t
	\end{aligned}
	\]
	Then $\beta_x$ is a $\mathscr{S}(\mathbb{R}, A; \alpha)$-$A$-$\hat{\otimes}$-bimodule homomorphism, and $\beta_y$ is a $A$-$\mathscr{S}(\mathbb{R}, A; \alpha)$-$\hat{\otimes}$-bimodule homomorphism. Moreover, we have
	\[
	\beta_x \circ \iota = \beta_y \circ \iota = \text{Id}_{\mathscr{S}(\mathbb{R}^2, A)_\alpha}.
	\]
\end{lemma}

\begin{proof}
	We'll start by proving that $\beta_x$ is well-defined: it is not entirely obvious from the construction that these integrals define functions which belong to $\mathscr{S}(\mathbb{R}^2, A)$. Let us prove that the corresponding integral over $\mathbb{R}$ equals zero, then we can use the vector-valued version of the Haramard's lemma to prove that the antiderivative lies in $\mathscr{S}(\mathbb{R}^2, A)$, as well.
	\[
	\begin{aligned}
	& \int_{\mathbb{R}} \left( \alpha_{t-x}(F(t, x+y - t)) - \varphi(x+y-t) \int_{\mathbb{R}} \alpha_{z-x}(F(z, x+y-z)) \text{d}z \right) \text{d}t = \\
	&= \int_{\mathbb{R}}  \alpha_{t-x}(F(t, x+y - t)) \text{d}t   - \int_{\mathbb{R}} \left( \varphi(x+y-t) \int_{\mathbb{R}} \alpha_{z-x}(F(z, x+y-z)) \text{d}z \right) \text{d}t = \\
	&= \int_{\mathbb{R}}  \alpha_{t-x}(F(t, x+y - t)) \text{d}t - \int_{\mathbb{R}} \varphi(x+y-t) \text{d}t  \int_{\mathbb{R}} \alpha_{z-x}(F(z, x+y-z)) \text{d}z = \\
	&= \int_{\mathbb{R}}  \alpha_{t-x}(F(t, x+y - t)) \text{d}t -  \int_{\mathbb{R}} \alpha_{z-x}(F(z, x+y-z)) \text{d}z = 0
	\end{aligned}
	\]
	Now we can prove that $\beta_x$ is a $\hat{\otimes}$-bimodule homomorphism. We notice that for every $H \in S, F \in \mathscr{S}(\mathbb{R}^2, A)$ we have
	\[
	\begin{aligned}
	\alpha_{t-x}((H \circ F)(t, x+y - t)) &= \alpha_{t-x}\left( \int_{\mathbb{R}} \alpha_{-t}(H(s)) F(t-s, x+y-t) \text{d}s \right) = \\ &= \int_{\mathbb{R}} \alpha_{-x}(H(s)) \alpha_{t-x}(F(t-s, x+y-t)) \text{d}s,
	\end{aligned}
	\]
	therefore, we have
	\[
	\small
	\begin{aligned}
		& \beta_x(H \circ F)(x, y) = \int_{-\infty}^x \left( \alpha_{t-x}(H \circ F(t, x+y - t)) - \varphi(x+y-t) \int_{\mathbb{R}} \alpha_{z-x}(H \circ F(z, x+y-z)) \text{d}z \right) \text{d}t = \\
		&= \int_{-\infty}^x \left( \int_{\mathbb{R}} \alpha_{-x}(H(s)) \alpha_{t-x}(F(t-s, x+y-t)) \text{d}s - \varphi(x+y-t) \int_{\mathbb{R}} \int_{\mathbb{R}} \alpha_{-x}(H(s)) \alpha_{z-x}(F(z-s, x+y-z)) \text{d}s \text{d}z \right) \text{d}t = \\
		&= \int_{-\infty}^x \left( \int_{\mathbb{R}} \alpha_{-x}(H(s)) \alpha_{t-x}(F(t-s, x+y-t)) \text{d}s - \varphi(x+y-t) \int_{\mathbb{R}} \alpha_{-x}(H(s)) \int_{\mathbb{R}}  \alpha_{z-x}(F(z-s, x+y-z)) \text{d}z \text{d}s \right) \text{d}t = \\
		&= \int_{-\infty}^x \left( \int_{\mathbb{R}} \alpha_{-x}(H(s)) \left(  \alpha_{t-x}(F(t-s, x+y-t)) - \varphi(x+y-t) \int_{\mathbb{R}}  \alpha_{z-x}(F(z-s, x+y-z)) \text{d}z \right)  \text{d}s \right) \text{d}t = \\
		&= \int_{\mathbb{R}} \left( \int_{-\infty}^x \alpha_{-x}(H(s)) \left(  \alpha_{t-x}(F(t-s, x+y-t)) - \varphi(x+y-t) \int_{\mathbb{R}}  \alpha_{z-x}(F(z-s, x+y-z)) \text{d}z \right)  \text{d}t \right) \text{d}s = \\
		&= \int_{\mathbb{R}} \alpha_{-x}(H(s)) \left( \int_{-\infty}^x \left(  \alpha_{t-x}(F(t-s, x+y-t)) - \varphi(x+y-t) \int_{\mathbb{R}}  \alpha_{z-x}(F(z-s, x+y-z)) \text{d}z \right)  \text{d}t \right) \text{d}s = \\
		&= \int_{\mathbb{R}} \alpha_{-x}(H(s)) \left( \int_{-\infty}^x \left(  \alpha_{t+s-x}(F(t, x+y-t-s)) - \varphi(x+y-t-s) \int_{\mathbb{R}}  \alpha_{z+s-x}(F(z, x+y-z-s)) \text{d}z \right)  \text{d}t \right) \text{d}s = \\
		&= \int_{\mathbb{R}} \alpha_{-x}(H(s)) \beta_x(F)(x-s,y) \text{d}s = (H \circ \beta_x(F))(x, y),
	\end{aligned}
	\]
	\[
	\begin{aligned}
	&\beta_x(F \circ a)(x, y) = \int_{-\infty}^x \left( \alpha_{t-x}((F \circ a)(t, x+y - t)) - \varphi(x+y-t) \int_{\mathbb{R}} \alpha_{z-x}((F \circ a)(z, x+y-z)) \text{d}z \right) \text{d}t = \\
	&= \int_{-\infty}^x \left( \alpha_{t-x}(F(t, x+y-t))\alpha_y(a) - \varphi(x+y-t) \int_{\mathbb{R}} \alpha_{z-x}(F(z, x+y-z))\alpha_y(a) \text{d}z \right) \text{d}t = \\
	&= \left( \int_{-\infty}^x \left( \alpha_{t-x}(F(t, x+y-t)) - \varphi(x+y-t) \int_{\mathbb{R}} \alpha_{z-x}(F(z, x+y-z)) \text{d}z \right) \text{d}t \right)  \alpha_y(a) = (\beta_x(F) \circ a)(x, y).
	\end{aligned}
	\]
	To check that $\beta_x$ is the right inverse to $\iota$, we have to assume that $F(x, y) = G(x)H(y)$. First of all, notice that
	\[
	\begin{aligned}
	& \frac{\text{d}}{\text{d} t} ( \alpha_{t-x}(F(t,x+y-t))) = \frac{\text{d}}{\text{d} t} (\alpha_{t-x}(G(t)) \alpha_{t-x}(H(x+y-t))) = \\
	&= (\alpha_{t-x}(G'(t))+\alpha'_{t-x}(G(t)))\alpha_{t-x}(H(x+y-t)) + \alpha_{t-x}(G(t))( \alpha'_{t-x}(H(x+y-t)) - \alpha_{t-x}(H'(x+y-t))) = \\
	&= \alpha_{t-x}(G'(t)H(x+y-t) + \alpha'_0(G(t))H(x+y-t) + G(t)\alpha'_0(H(x+y-t)) - G(t)H'(x+y-t)) \stackrel{\ref{eq4}}{=}\\
	&\stackrel{\ref{eq4}}{=} \alpha_{t-x}(T^{-1}((TG)')(t) H(x+y-t) - G(t) T((T^{-1}H)')(x+y-t)) = \alpha_{t-x} (\iota(F)(t, x+y-t)).
	\end{aligned}
	\]
	Therefore, we have
	\[
	\begin{aligned}
		&(\beta_x \circ \iota(F))(x, y) = \int_{-\infty}^x \left( \alpha_{t-x}(\iota(F)(t, x+y - t)) - \varphi(x+y-t) \int_{\mathbb{R}} \alpha_{z-x}(\iota(F)(z, x+y-z)) \text{d}z \right) \text{d}t = \\
		&= \int_{-\infty}^x \left(  \frac{\text{d}}{\text{d} t} ( \alpha_{t-x}(F(t,x+y-t))) - \varphi(x+y-t) \int_{\mathbb{R}} \frac{\text{d}}{\text{d} z} ( \alpha_{z-x}(F(z,x+y-z))) \text{d}z\right) \text{d}t = F(x, y).
	\end{aligned}
	\]
	The necessary computations for $\beta_y$ are, essentially, the same.
\end{proof}


By combining the Lemmas \ref{lemma2.1} -- \ref{lemma2.6}, we can formulate the following theorem:

\begin{theorem}
	\label{main theorem}
	Let $A$ be a self-induced Fr\'echet-Arens-Michael algebra with a smooth $m$-tempered action $\alpha$ of $\mathbb{R}$ on $A$. Then the following diagram is commutative, moreover, the rows are short exact sequences of $\mathscr{S}(\mathbb{R}, A; \alpha)$-bimodules which split in the categories $\mathscr{S}(\mathbb{R}, A; \alpha)$-$A$-Mod$(\textbf{Fr})$ and $A$-$\mathscr{S}(\mathbb{R}, A; \alpha)$-Mod$(\textbf{Fr})$:
	\[
	\begin{tikzcd}[column sep = scriptsize]
	0 \arrow[r] & \mathscr{S}(\mathbb{R}^2, A)_\alpha \arrow[r, "\iota"] & \mathscr{S}(\mathbb{R}^2, A)_\alpha \arrow[r, "\pi"] & \mathscr{S}(\mathbb{R}, A; \alpha)_\alpha \arrow[r] & 0, \\
	0 \ar[r] & \mathscr{S}(\mathbb{R}, A; \alpha)_\alpha \hat{\otimes}_A  \mathscr{S}(\mathbb{R}, A; \alpha)_\alpha \ar[r, "j"] \isoarrow{u} & \mathscr{S}(\mathbb{R}, A; \alpha)_\alpha \hat{\otimes}_A \mathscr{S}(\mathbb{R}, A; \alpha)_\alpha \ar[r, "m"] \isoarrow{u} & \mathscr{S}(\mathbb{R}, A; \alpha)_\alpha \ar[u, "\text{Id}"] \ar[r] & 0,
	\end{tikzcd}
	\] 
	where 
	\[
	\begin{aligned}
	\iota(F)(x, y) &= \left( \frac{\partial F }{\partial x} - \frac{\partial F}{\partial y}  \right)(x, y) + \alpha'_0(F(x, y)) & \text{ for any } F \in \mathscr{S}(\mathbb{R}^2, A) \\
	\pi(F)(x) &= \int_{\mathbb{R}} \alpha_y ( F(y,x-y) ) \text{d} y & \text{ for any } F \in \mathscr{S}(\mathbb{R}^2, A) \\
	j(F \otimes G) &= F' \otimes G - F \otimes T((T^{-1}G)') &\text{ for any } F, G \in \mathscr{S}(\mathbb{R}, A; \alpha) \\
	m(F \otimes G) &= F *_\alpha G & \text{ for any } F, G \in \mathscr{S}(\mathbb{R}, A; \alpha).
	\end{aligned}
	\]
\end{theorem}
\begin{proof}
	In the previous lemmas we have constructed the sections $\rho_x, \rho_y, \beta_x, \beta_y$. The only thing that is left to check that $\iota \circ \beta_x + \rho_x \circ \pi = \iota \circ \beta_y +  \rho_y \circ \pi = \text{Id}_{\mathscr{S}(\mathbb{R}^2, A)_\alpha}$, then we use \cite[Proposition 3.1.8]{Helem1986}. 
	
	For any $F(x, y) \in \mathscr{S}(\mathbb{R}^2, A)_\alpha$ we have
	\[
	\begin{aligned}
	&(\iota \circ \beta_x)(F)(x, y) = \iota\left( \int_{-\infty}^x \left( \alpha_{t-x}(F(t, x+y-t))  - \varphi(x+y-t) \int_{\mathbb{R}} \alpha_{z-x}(F(z, x+y-z)) \text{d}z \right) \text{d}t \right) = \\
	&= F(x, y) - \varphi(y) \int_{\mathbb{R}} \alpha_{z-x}(F(z, x+y-z)) \text{d}z + \\ 
	&+ \int_{-\infty}^x \left( \frac{\partial }{\partial x} - \frac{\partial }{\partial y} \right)  \left( \alpha_{t-x}(F(t, x+y-t))  - \varphi(x+y-t) \int_{\mathbb{R}} \alpha_{z-x}(F(z, x+y-z)) \text{d}z \right) \text{d}t + \\
	&+ \int_{-\infty}^x \left( \alpha'_{t-x}(F(t, x+y-t))  - \varphi(x+y-t) \int_{\mathbb{R}} \alpha'_{z-x}(F(z, x+y-z)) \text{d}z \right) \text{d}t = \\
	&= F(x, y) - \varphi(y) \int_{\mathbb{R}} \alpha_{z-x}(F(z, x+y-z)) \text{d}z + \\ 
	&+ \int_{-\infty}^x \left( -\alpha'_{t-x}(F(t, x+y-t)) + \varphi(x+y-t) \int_{\mathbb{R}} \alpha'_{z-x}(F(z, x+y-z)) \text{d}z \right) \text{d}t + \\
	&+ \int_{-\infty}^x \left( \alpha'_{t-x}(F(t, x+y-t))  - \varphi(x+y-t) \int_{\mathbb{R}} \alpha'_{z-x}(F(z, x+y-z)) \text{d}z \right) \text{d}t = \\
	&= F(x, y) - \varphi(y) \int_{\mathbb{R}} \alpha_{z-x}(F(z, x+y-z)) \text{d}z,
	\end{aligned}
	\]
	\[
	\begin{aligned}
	(\rho_x \circ \pi)(F)(x) &= \varphi(y) \alpha_{-x}(\pi(F)(x+y)) = \varphi(y) \int_{\mathbb{R}} \alpha_{z-x}(F(z, x+y-z)) \text{d}z,
	\end{aligned}
	\]
	therefore, we have
	\[
	\iota \circ \beta_x + \rho_x \circ \pi(F)(x, y) = F(x, y).
	\]
	The argument for $\iota \circ \beta_y +  \rho_y \circ \pi$ is similar.
\end{proof}

In the case $G = \mathbb{T}$ we obtain the following theorem.

\begin{theorem}
	\label{main theorem2}
	Let $A$ be a self-induced Fr\'echet-Arens-Michael algebra with a smooth $m$-tempered action $\alpha$ of $\mathbb{T}$ on $A$. Then the following diagram is commutative, moreover, the rows are short exact sequences of $C^{\infty}(\mathbb{T}, A; \alpha)$-bimodules which split in the categories $C^{\infty}(\mathbb{T}, A; \alpha)$-$A$-Mod$(\textbf{Fr})$ and $A$-$C^{\infty}(\mathbb{T}, A; \alpha)$-Mod$(\textbf{Fr})$:
	\[
	\begin{tikzcd}[column sep = scriptsize]
	0 \arrow[r] & C^{\infty}(\mathbb{T}^2, A)_\alpha \arrow[r, "\iota"] & C^{\infty}(\mathbb{T}^2, A)_\alpha \arrow[r, "\pi"] & C^{\infty}(\mathbb{T}, A; \alpha)_\alpha \arrow[r] & 0, \\
	0 \ar[r] & C^{\infty}(\mathbb{T}, A; \alpha)_\alpha \hat{\otimes}_A  C^{\infty}(\mathbb{T}, A; \alpha)_\alpha \ar[r, "j"] \isoarrow{u} & C^{\infty}(\mathbb{T}, A; \alpha)_\alpha \hat{\otimes}_A C^{\infty}(\mathbb{T}, A; \alpha)_\alpha \ar[r, "m"] \isoarrow{u} & C^{\infty}(\mathbb{T}, A; \alpha)_\alpha \ar[u, "\text{Id}"] \ar[r] & 0,
	\end{tikzcd}
	\] 
	where 
	\[
	\begin{aligned}
	\iota(F)(x, y) &= \left( \frac{\partial F }{\partial x} - \frac{\partial F}{\partial y}  \right)(x, y) + \alpha'_0(F(x, y)) & \text{ for any } F \in C^{\infty}(\mathbb{T}^2, A) \\
	\pi(F)(x) &= \int_{\mathbb{T}} \alpha_y ( F(y,x-y) ) \text{d} y & \text{ for any } F \in C^{\infty}(\mathbb{T}^2, A) \\
	j(F \otimes G) &= F' \otimes G - F \otimes T((T^{-1}G)') &\text{ for any } F, G \in C^{\infty}(\mathbb{T}, A; \alpha) \\
	m(F \otimes G) &= F *_\alpha G & \text{ for any } F, G \in C^{\infty}(\mathbb{T}, A; \alpha).
	\end{aligned}
	\]
\end{theorem}

\section{Obtaining upper estimates for homological dimensions of smooth crossed products by \texorpdfstring{$\mathbb{R}$}{R} and \texorpdfstring{$\mathbb{T}$}{T}}

\textbf{Remark.} Again, we provide the proofs only for the case $G = \mathbb{R}$, but the same arguments work for $G = \mathbb{T}$, as well.

Here we adapt the arguments in \cite{2017arXiv171206177K}, which were used to obtain the upper estimates to the non-unital case.

\begin{definition}
	\label{non-unital}
	Let $A$ be a $\hat{\otimes}$-algebra. Then a $\hat{\otimes}$-algebra $S$ together with a $A$-$\hat{\otimes}$-bimodule structure is called an $A$-$\hat{\otimes}$-algebra if $S \in A$-\textbf{mod}-$S$ and $S \in S$-\textbf{mod}-$A$.
\end{definition}

This definition works as expected in the unital case.

\begin{proposition}
	Let $A$ be a unital $\hat{\otimes}$-algebra. A $A$-$\hat{\otimes}$-bimodule structure on a unital $A$-$\hat{\otimes}$-algebra $S$ is uniquely defined by a (unital) algebra homomorphism $\eta : A \rightarrow S$:
	\[
	a \circ s = \eta(a) s, \quad s \circ a = s \eta(a)
	\]
	for every $a \in A, s \in S$.
\end{proposition}
\begin{proof}
	Define $\eta$ as follows:	$\eta(a) = a \circ 1_S = 1_S \circ a.$ It is easy to see that $\eta$ is an algebra homomorphism. Also, we have
	\[
	\eta(a) s = (a \circ 1_S)s = a \circ (1_Ss) = a \circ s,
	\]
	\[
	s \eta(a) = s (a \circ 1_S) = s (1_S \circ a) = (s 1_S) \circ a = s \circ a
	\]
	for any $a\in A$, $s \in S$.
\end{proof}

As a corollary from Lemma \ref{lemma2.1} we get that the $A$-$\hat{\otimes}$-bimodule structure on $S_\alpha$ makes $\mathscr{S}(\mathbb{R}, A; \alpha)$ into a $A$-$\hat{\otimes}$-algebra.

\begin{proposition}
	\label{smooth extens is left and right free}
	Let $A$ be a Fr\'echet-Arens-Michael algebra, and let $\alpha$ be a $m$-tempered action of $\mathbb{R}$ on $A$. Consider the following multiplication on $\mathscr{S}(\mathbb{R}, A)$:
	\[
	(f *_\alpha' g)(x) = \int_\mathbb{R} \alpha_{-y}(f(x-y)) g(y) \text{d}y.
	\]
	Then the following locally convex algebra isomorphism takes place:
	\[
	i : \mathscr{S}(\mathbb{R}, A; \alpha) \rightarrow (\mathscr{S}(\mathbb{R}, A), *_\alpha'), \quad i(f)(x) = \alpha_{-x} (f(x)).
	\]
\end{proposition}

\begin{proof}
	The mapping $i$ is, obviously, a topological isomorphism of locally convex spaces. Now notice that
	\[
	\begin{split}
	(i(f) *' i(g))(x) 
	&= \int_\mathbb{R} \alpha_{-y} (i(f)(x-y)) i(g)(y) \text{d}y = \int_\mathbb{R} \alpha_{-x} (f(x-y)) \alpha_{-y}(g(y)) \text{d}y = \\
	&= \int_\mathbb{R} \alpha_{-x} (f(-y)) \alpha_{-x-y}(g(y+x)) \text{d} y = \alpha_{-x} \left( \int_\mathbb{R} f(-y) \alpha_{-y}(g(y+x)) \text{d}y \right) = \\
	&= i(f * g)(x), 
	\end{split}
	\]
	therefore, $i$ is an algebra homomorphism.
\end{proof}
\begin{corollary}
	\label{corollary1}
	Define the $A$-$\hat{\otimes}$-bimodule and $(\mathscr{S}(\mathbb{R}, A), *_\alpha')$-$\hat{\otimes}$-bimodule ${}_{\alpha^{-1}} \mathscr{S}(\mathbb{R}, A; \alpha) = {}_{\alpha^{-1}} S$ as follows: ${}_{\alpha^{-1}} S$ coincides with $\mathscr{S}(\mathbb{R}, A)$ as a LCS, and
	\[
	\begin{aligned}
	(F \circ a)(x) &= F(x) a, & \quad a \circ F(x) &= \alpha_{-x}(a) F(x) \quad & \text{ for any } &a \in A, F \in {}_{\alpha^{-1}} S \\
	(F \circ G)(x) &= (F *_\alpha' G)(x), & \quad (G \circ F)(x) &= (G *_\alpha' F)(x) \quad & \text{ for any } &F \in {}_{\alpha^{-1}} S, G \in (\mathscr{S}(\mathbb{R}, A), *_\alpha').
	\end{aligned}
	\]
	Then the map 
	\[
	S_\alpha \longrightarrow {}_{\alpha^{-1}} S, \quad F(x) \longmapsto \alpha_{-x}(F(x)),
	\]
	is an isomorphism of $A$-$\hat{\otimes}$-bimodules.	As a corollary, $S_\alpha$ is projective as a left and right $A$-$\hat{\otimes}$-bimodule.
\end{corollary}

\begin{definition}
	Let $A$ be a $\hat{\otimes}$-algebra. A left $A$-$\hat{\otimes}$-module $M$ is called essential, if the canonical morphism $\eta_M : A \hat{\otimes}_A M \longrightarrow M$ is an isomorphism of left $A$-$\hat{\otimes}$-modules.
\end{definition}

\begin{lemma}
	\label{essential}
	Let $A$ be a self-induced Fr\'echet-Arens-Michael algebra together with an $m$-tempered $\mathbb{R}$-action $\alpha$ of $\mathbb{R}$. Then the module $\mathscr{S}(\mathbb{R}, A; \alpha)_\alpha$ is an essential left and right $A$-$\hat{\otimes}$-module.
\end{lemma}
\begin{example}
	If $A$ is a self-induced $\hat{\otimes}$-algebra, then for any left $A$-$\hat{\otimes}$-module $M$ the module $A \hat{\otimes}_A M$ is essential.
\end{example}
\begin{proof}
	Recall that the mapping $i : A \hat{\otimes} \mathscr{S}(\mathbb{R})_\alpha \longrightarrow \mathscr{S}(\mathbb{R}, A; \alpha)$, $i(a \otimes f)(x) = f(x)a$, is an isomorphism of left $A$-$\hat{\otimes}$-modules. Therefore, we can write the following composition of isomorphisms:
	\[
	A \hat{\otimes}_A \mathscr{S}(\mathbb{R}, A; \alpha)_\alpha \xrightarrow{i^{-1}} A \hat{\otimes}_A A \hat{\otimes} \mathscr{S}(\mathbb{R}) \xrightarrow{m} A \hat{\otimes} \mathscr{S}(\mathbb{R}) \xrightarrow{i} \mathscr{S}(\mathbb{R}, A; \alpha)_\alpha.
	\]
	Notice that $i^{-1} \circ m \circ i(a \otimes f(x)b) = f(x)ab = \eta_M(a \otimes f(x)b)$, so $\eta_M$ coincides with $i^{-1} \circ m \circ i$ on a dense subset, therefore, $\eta_M$ is an isomorphism of left $A$-$\hat{\otimes}$-modules.
	The same argument shows that ${}_{\alpha^{-1}} \mathscr{S}(\mathbb{R}, A; \alpha)$ is an essential right $A$-$\hat{\otimes}$-module, but ${}_{\alpha^{-1}} \mathscr{S}(\mathbb{R}, A; \alpha) \cong \mathscr{S}(\mathbb{R}, A; \alpha)_\alpha$.
\end{proof}

%

\begin{lemma}
	\label{projective1}
	Let $A$ be a $\hat{\otimes}$-algebra and let $S$ be an $A$-$\hat{\otimes}$-algebra.
	\begin{enumerate}[label=(\arabic*)]
		\item Let $X$ be a projective right $A$-$\hat{\otimes}$-module. Then the module $X \hat{\otimes}_A S$ is a projective right $S$-$\hat{\otimes}$-module. Similarly, if $X$ is a projective left $A$-$\hat{\otimes}$-module, then $S \hat{\otimes}_A X$ is a projective left $S$-$\hat{\otimes}$-module.
		\item Let $X$ be a projective $A$-$\hat{\otimes}$-bimodule. Then the module $S \hat{\otimes}_A X \hat{\otimes}_A S$ is a projective $S$-$\hat{\otimes}$-bimodule.
	\end{enumerate}
\end{lemma}
\begin{proof}
	\begin{enumerate}[label=(\arabic*)]
		\item The module $X$ is projective, therefore, there is a retraction $\sigma : E \hat{\otimes} A_+ \longrightarrow X$. But then the map
		\[
		E \hat{\otimes} S_+ \longrightarrow E \hat{\otimes} S \cong E \otimes A_+ \hat{\otimes}_A S \longrightarrow X \hat{\otimes}_A S  
		\]
		is a composition of retractions, and retracts of free modules are projective. Proof for the left modules is similar.
		\item The bimodule $X$ is projective, therefore, there is a retraction $\sigma : A_+ \hat{\otimes} E \hat{\otimes} A_+ \longrightarrow X$. Then the map
		\[
		S \hat{\otimes} E \hat{\otimes} S \longrightarrow S \hat{\otimes}_A A_+ \hat{\otimes} E \hat{\otimes} A_+ \hat{\otimes}_A S \longrightarrow S \hat{\otimes}_A X \hat{\otimes}_A S
		\]
		is a composition of retractions. Due to \cite[Proposition 4.1.4]{Helem1986}, the $S$-$\hat{\otimes}$-bimodule $S \hat{\otimes} E \hat{\otimes} S$ is projective, and retracts of projective modules are projective.
	\end{enumerate}
\end{proof}

\begin{lemma}
	\label{admis1}
	Let $A$ be a $\hat{\otimes}$-algebra. Also let $\{M, d \}$ denote an admissible sequence of right $A$-$\hat{\otimes}$-modules. If $X$ is a projective left $A$-$\hat{\otimes}$-module, then the complex $\{ M \hat{\otimes}_A X, d \otimes \text{Id} \}$ splits in \textbf{LCS}.
	
	Similarly, for every projective right $A$-$\hat{\otimes}$-module $X$ the complex $\{ X \hat{\otimes} M, d \otimes \text{Id} \}$ splits in \textbf{LCS}.
\end{lemma}

\begin{proof}
	If $X$ were a free left $A$-$\hat{\otimes}$-module, then the statement of the lemma would follow from the canonical isomorphism $M \hat{\otimes}_A A_+ \hat{\otimes} E \cong M \hat{\otimes} X$ for some $E \in \textbf{LCS}$. However, a retract of an admissible sequence is admissible, as well.
\end{proof}

\begin{lemma}
	\label{l3.4}
	Let $A$ be a $\hat{\otimes}$-algebra and let $S$ be an $A$-$\hat{\otimes}$-algebra, which is projective as a left $A$-$\hat{\otimes}$-module. Then we have 
	\[
	\text{dg}_{S^{op}}(M \hat{\otimes}_A S) \le \text{dg}_{A^{op}}(M).
	\]
	If $S$ is projective as a right $A$-$\hat{\otimes}$-module, then
	\[
	\text{dg}_S(S \hat{\otimes}_A M) \le \text{dg}_A(M).
	\]
\end{lemma}
\begin{proof}
	Suppose we have a projective resolution of $M$ in $\textbf{mod}$-$A$:
	\[
	0 \longleftarrow M \stackrel{d_0}{\longleftarrow} P_0 \stackrel{d_1}{\longleftarrow} \dots \longleftarrow P_n \longleftarrow 0 \longleftarrow \dots
	\]
	Then due to Lemma \ref{projective1} and \ref{admis1} the following sequence is the projective resolution for $M \hat{\otimes}_A S$ in $\textbf{mod}$-$S$:
	\[
	0 \longleftarrow M\hat{\otimes}_A S \stackrel{d_0 \otimes \text{Id}}{\longleftarrow} P_0 \hat{\otimes}_A S \stackrel{d_1 \otimes \text{Id}}{\longleftarrow} \dots \longleftarrow P_n \hat{\otimes}_A S \longleftarrow 0 \longleftarrow \dots
	\]
	Therefore, $\text{dg}_{S^{op}}(M \hat{\otimes}_A S) \le \text{dg}_{A^{op}}(M)$.
\end{proof}


\begin{lemma}
	\label{estimate for essential modules}
	Let $A$ be a self-induced Fr\'echet-Arens-Michael algebra together with a smooth $m$-tempered $\mathbb{R}$-action $\alpha$. Set $S = \mathscr{S}(\mathbb{R}, A; \alpha)$. 
	For any right $S$-$\hat{\otimes}$-module $M$ we have the following estimate:
	\[
	\text{dh}_{S^{\text{op}}}(M \hat{\otimes}_S S_\alpha) \le \text{dh}_{A^{\text{op}}}(M \hat{\otimes}_S S_\alpha) + 1 \le \text{dg}(A^{\text{op}}) + 1.
	\]
	And for any left $S$-$\hat{\otimes}$-module $M$ we have
	\[
	\text{dh}_{S}(S_\alpha \hat{\otimes} M) \le \text{dh}_{A}(S_\alpha \hat{\otimes} M) + 1 \le \text{dg}(A) + 1.
	\]
\end{lemma}
\begin{proof}
	Due to the Theorem \ref{main theorem} we have the following sequence:
	\begin{equation}
	\label{yeah}
		0 \longrightarrow S_\alpha \hat{\otimes}_A S_\alpha \longrightarrow S_\alpha \hat{\otimes}_A S_\alpha \longrightarrow S_\alpha \longrightarrow 0.
	\end{equation}
	By applying the functor $M \hat{\otimes}_S (-)$ to \ref{yeah}, we get
	\[
	0 \longrightarrow M \hat{\otimes}_S S_\alpha \hat{\otimes}_A S_\alpha \longrightarrow M \hat{\otimes}_S S_\alpha \hat{\otimes}_A S_\alpha \longrightarrow M \hat{\otimes}_S S_\alpha \longrightarrow 0.
	\]
	Obviously, this sequence is admissible, therefore, we can apply Lemma \ref{l3.4}, because $S_\alpha$ is a projective $A$-$\hat{\otimes}$-module (Corollary \ref{corollary1}).
	\[
	\text{dh}_{S^{\text{op}}}(M \hat{\otimes}_S S_\alpha) \le \text{dh}_{S^{\text{op}}}(M \hat{\otimes}_S S_\alpha \hat{\otimes}_A S_\alpha) + 1 \stackrel{L\ref{l3.4}}{\le} \text{dh}_{A^{\text{op}}}(M \hat{\otimes}_S S_\alpha) + 1 \le \text{dg}(A^{\text{op}}) + 1.
	\]
\end{proof}

So, we have just obtained the upper bound for projective dimension of essential modules. To obtain an estimate for an arbitrary right $S$-$\hat{\otimes}$-module, we use the method, described in the Lemmas 1-3 of the paper \cite{ogneva1984homological}.

\begin{theorem}{\cite[Theorem 5.2.1]{Helem1986}}
	Let $A$ be a $\hat{\otimes}$-algebra and let $X$ be a left $A$-$\hat{\otimes}$ module. Then the following complex is admissible:
	\[
	0 \longleftarrow X \longleftarrow (A_+ \hat{\otimes} X) \oplus (A \hat{\otimes}_A X) \xleftarrow{\delta_0} (A_+ \hat{\otimes} (A \hat{\otimes}_A X)) \oplus A \hat{\otimes} X \xleftarrow{\delta_1} A \hat{\otimes} (A \hat{\otimes}_A X) \longleftarrow 0.
	\]
	Moreover, this sequence is isomorphic to the tensor product of the following short admissible complexes:
	\begin{equation}
		0 \longleftarrow \text{Im}\, \delta_0 \xleftarrow{\delta_0} (A_+ \hat{\otimes} (A \hat{\otimes}_A X)) \oplus A \hat{\otimes} X \xleftarrow{\delta_1} A \hat{\otimes} (A \hat{\otimes}_A X) \longleftarrow 0,
	\end{equation}
	\begin{equation}
		0 \longleftarrow X \longleftarrow (A_+ \hat{\otimes} X) \oplus (A \hat{\otimes}_A X) \hookleftarrow \text{Im}\, \delta_0 \longleftarrow 0.
	\end{equation}
\end{theorem}

Notice that the modules $(A_+ \hat{\otimes} (A \hat{\otimes}_A X)) \oplus A \hat{\otimes} X$ and $A \hat{\otimes} (A \hat{\otimes}_A X)$ are projective left $A$-modules, therefore, $\text{dh}_A(\text{Im}\, \delta_0) \le 1$.

But then we also have
\begin{equation}
	\label{estimate}
	\text{dh}_A(X) \le \max \{ \text{dh}_A((A_+ \hat{\otimes} X) \oplus (A \hat{\otimes}_A X)), \text{dh}_A(\text{Im}\, \delta_0) + 1 \} \le \max \{ \text{dh}_A(A \hat{\otimes}_A X), 2\},
\end{equation}
for any $\hat{\otimes}$-algebra $A$ and a left $A$-$\hat{\otimes}$-module $X$.
Combining \eqref{estimate} with the Lemma \ref{estimate for essential modules}, we get the following: for every left $S$-$\hat{\otimes}$-module $M$ we have
\[
\text{dh}_{S}(M) \stackrel{\eqref{estimate}}{\le} \max \{\text{dh}_{S}(S_\alpha \hat{\otimes}_S M), 2 \} \stackrel{L\ref{estimate for essential modules}}{\le} \max \{ \text{dgl}(A) + 1, 2\} = \max \{ \text{dgl}(A), 1\} + 1.
\]

\begin{theorem}
	\label{main theorem3}
	Let $A$ be a self-induced Fr\'echet-Arens-Michael algebra equipped with a smooth $m$-tempered $\mathbb{R}$-action $\alpha$. Then the following estimate takes place:
	\[
	\text{dgl}(\mathscr{S}(\mathbb{R}, A; \alpha)) \le \max \{ \text{dgl}(A), 1\} + 1.
	\]
\end{theorem}
The same result holds for $G = \mathbb{T}$:
\begin{theorem}
	\label{main theorem4}
	Let $A$ be a self-induced Fr\'echet-Arens-Michael algebra equipped with a smooth $m$-tempered $\mathbb{T}$-action $\alpha$. Then the following estimate takes place:
	\[
	\text{dgl}(C^\infty(\mathbb{T}, A; \alpha)) \le \max \{ \text{dgl}(A), 1\} + 1.
	\]
\end{theorem}

	\printbibliography
	\Addresses	
\end{document}